\documentclass[10pt,twoside]{amsart}

\usepackage{amssymb}
\usepackage[english]{babel}
\usepackage{graphicx,epstopdf,epsfig}
\usepackage{algorithm}
\usepackage{algpseudocode}
\usepackage{amsfonts,epsfig,fancyhdr,graphics, hyperref,amsmath,amssymb}

\newcommand{\Names}{Assaf Goldberger, Radel Ben-Av, Giora Dula, and Yoseph Strassler}
\newcommand{\Title}{Constructing, Classifying and Studying the Space of Small Integer Weighing Matrices}

\DeclareMathOperator{\Aut}{\mathrm{Aut}}
\DeclareMathOperator*{\PermAut}{\mathrm{PermAut}}
\DeclareMathOperator{\Mon}{\mathrm{Mon}}
\DeclareMathOperator*{\Isom}{\mathrm{Isom}}
\DeclareMathOperator*{\PermIsom}{\mathrm{PermIsom}}
\DeclareMathOperator*{\Perm}{\mathrm{Perm}}
\DeclareMathOperator{\SAut}{\mathrm{SAut}}
\DeclareMathOperator{\TAut}{\mathrm{TAut}}
\DeclareMathOperator{\diag}{\mathrm{diag}}

\newtheorem{theorem}{Theorem}[section]
\newtheorem{remark}[theorem]{Remark}
\newtheorem{example}[theorem]{Example}
\newtheorem{definition}[theorem]{Definition}
\newtheorem{lemma}[theorem]{Lemma}
\newtheorem{corollary}[theorem]{Corollary}
\newtheorem{proposition}[theorem]{Proposition}
\newcommand{\ZZ}{\mathbb Z}

\newcommand{\be}{\begin{equation}}
\newcommand{\ee}{\end{equation}}
\newcommand{\nn}{\langle n\rangle}
\newcommand{\FF}{\mathbb{F}}
 
\begin{document}

\bibliographystyle{plain}

\setcounter{page}{1}

\thispagestyle{empty}

 \title{\Title}

\author{
Assaf Goldberger
\and
Radel Ben-Av
\and 
Giora Dula
\and
{Yoseph Strassler}
}


\markboth{\Names}{\Title}

\begin{abstract}
Integer weighing matrices (IW-matrices for short) are integer valued orthogonal square matrices. One usecase of these is to create classical weighing matrices with various block structures. In this paper we study and classify the space $IW(n,k)$ of the integer weighing matrices of small size $n\times n$ and weight $k$. Our classification includes a full list of all inequivalent matrices up to Hadamard equivalence and automorphism groups \cite{repo_sym2025}. We then continue to a secondary classification of the symmetric and antisymmetric IW up to symmetric Hadamard equivalence. We apply this to the case of projective space weighing matrices. Next we use the classification to count the cardinality of the spaces of all $IW(n,k)$ as well as the symmetric and anti-symmetric subspace. We supply practical algorithms and implement them in \texttt{Sagemath} \cite{sagemath}. Finding an (anti-)symmetric IW matrix in a given Hadamard class can be done for significantly higher orders. In particular we solve some open cases: Symmetric $W(23,16)$, $W(28,25)$ and $W(30,17)$, and an anti-symmetric $W(28,25)$. We conclude by showing a detailed classification of $IW(7,25)$. We have also improved the \texttt{NSOKS} \cite{riel2006nsoks} algorithm to find all possible representations of an integer $k$ as a sum of $n$ integer squares.   
\end{abstract}

\maketitle




\section{Introduction} \label{intro-sec}

Orthogonal matrices play a pivotal role in various fields of mathematics and engineering, particularly in linear algebra,
numerical analysis, and quantum mechanics. These matrices play an important role also in coding theory, cryptography, and combinatorial design.
In this paper, we delve into the problem of characterizing
orthogonal matrices with integer entries, where all vectors are of given uniform magnitude.

The challenge lies in ensuring that the orthogonality condition 
is met while restricting the entries to integers, a constraint that makes the solution space discrete and finite. However, it is still not easy to 
search due to it's high cardinality.

Let us define $PIW(m,n,k) = \{ \ P\ |\ P\in  \ZZ^{m\times n}\ ,\ PP^\top =kI\}$,   $\ZZ$  is the ring of integers, and let $IW(n,k)=PIW(n,n,k)$. Classical weighing matrices $W(n,k)$ are the subset of $IW(n,k)$ of matrices over $\{-1,0,1\}$. Weighing matrices have been extensively investigated over the past few decades \cite{chan1986inequivalent,Seberry2017OrthogonalDesigns,craigen1992constructions,SeberryYamada2020,Hadi:Sho:BTR:JACO:2022,Hadi:Thomas:Sho:JCTA:2022,M_M_Tan:DCC:2018}. 
A particular interesting subcase are Hadamard matrices $H(n)=W(n,n)$. Hadamard proved that $H(n)$ is nonempty for $n>2$ only if $n$ is divisible by 4. The Hadamard conjecture is that $H(n)$ is nonempty for all $n$ divisible by $4$. This is extended to a conjecture on weighing matrices, postulating that $W(4l,k)$ is nonempty for every $l$ and every $k\le 4l$.
For a (somewhat outdated) summary of methods and existence tables, see \cite[Chap. V]{colbourn2006handbook}.\\

A circulant matrix is square $n\times n$ matrix $C$ such that $C_{i,j}=f(i-j)$ where $f:\ZZ/n \to \mathbb C$ and the difference is taken modulo $n$. Circulant weighing matrices, denoted by $CW(n,k)$ exist only if $k$ is a perfect square, and under this assumption, there is a variety of solutions. For example, the weight $k=9$ has been fully classified, see \cite{ANG20082802}. Integral circulant weighing matrices, denoted by $ICW(n,k)$, are a stepping stone for the construction of circulant or general classical weighing matrices, see e.g. \cite{gutman2009circulant}. In a similar fashion, we may try to construct a weighing matrix $\displaystyle M=( C_{ij}) \in W(\displaystyle n_1n_2,k)$, with circulant blocks $C_{i,j}$ of size $n_2\times n_2$. If we replace each circulant block $C_{i,j}$ with its its first row sum $c_{i,j}$, the resulting matrix $S=(c_{i,j})$ is in $IW(n_1,k)$. The authors \cite{BENAV2024113908}, have used this method to construct $W(25,16)$ from a certain $IW(5,16)$. (see also \cite{munemasa2017weighing} for the matrix).
A classification of the space of $IW(7,25)$ is thus an appealing direction towards a construction of the unknown $W(35,25)$. Other special instances of this method are the doubling \cite{Seberry2005} and Wiliamson \cite{Williamson1944} constructions. It is also worthwhile to mention the concept of ``multilevel Hadamard matrices'' \cite{parker2011multilevel}, which has its separate motivation, and is a special case of $IW(n,k)$. Integral weighing matrices have also connections to number theory and lattice theory. IW-matrices are a special case of the problem of classifying integral quadratic forms, involving some deep topics such as local-to-global principles, genera and class groups, see e.g. \cite{Conway1999}.\\

Partial Hadamard and weighing matrices appear in the literature in many places.
In \cite{Goldberger2022} a practical method is given for completion of a $PIW(n/2,n,n)$ to a full Hadamard matrix $W(n,n)$. The paper \cite{Craigen2013CirculantPH} studies circulant partial Hadamard matrices in effort to shed more light on the \emph{circulant Hadamard conjecture}. Partial Hadamard matrices are also studied in relation to quantum groups \cite{BANICA2015364}, and fMRI technology \cite{cheng2015optimal}.
\begin{definition}
Two $PIW(m,n,k)$ matrices are \emph{Hadamard Equivalent} (or H-equivalent) if one can be obtained from the other using a sequence of the following operations:

\begin{enumerate}
    \item Multiplying a row or a column by -1, and
    \item Swapping two rows or columns.
\end{enumerate}
Note that $PIW(m,n,k)$ is closed under H-equivalence. In the case of $IW(n,k)$ we can add another operation:

\begin{enumerate}
    \setcounter{enumi}{2}
    \item Transpose the matrix.
\end{enumerate}
\end{definition}
Adding this operation we obtain an extension of the previous equivalence relation, which we call \emph{transpose-Hadamard} equivalence relation (or TH-equivalence). We denote H and TH equivalences by the relation $A\sim_HB$ and $A\sim_{TH}B$ respectively.\\ 

The classification of classical weighing matrices up to H-equivalence is a an important problem in the areas of combinatorial design theory and coding theory. 
The relation to coding theory comes via the classification of self-dual codes. 
The early work of \cite{MallowsPlessSloane1976} classifies self dual ternary codes of length up to $12$. A classification of all $W(n,k)$ for $n\le 11$ and $k\le 5$ appears in \cite{ChanRodgerSeberry1986}. Ohmori \cite{ohmori1988classifications,ohmori1993classification} classifies all weighing matrices of size $12$ and all $W(13,9)$. Harada and Munemasa \cite{Harada2011} have classified all $W(17,9)$ and found $2360$ classes. A recent unpublished work of Lampio and Ganzhinov \cite{https://us.ticmeet.com/assets/archivos/d6f1d9b8-d39f-4888-925a-7eb81c4905cc/Lampio.pdf} classifies completely weighing matrices of orders $16,18,19,20,21$ and partially $22,24,28$.
\\

In this paper we extend the classification problem to IW-matrices. We treat both H and TH-equivalences. We also compute the automorphism group of each equivalence class. In addition, we establish algorithms for finding and classifying all symmetric and anti-symmetric $IW(n,k)$ up to symmetric Hadamard equivalence, which is a finer equivalence relation. We utilize the automorphism groups compute the total number of all $IW(n,k)$ for small $n,k$ as well as the cardinality of symmetric and anti-symmetric IW-matrices. Our methods are well suited also for the case of classical weighing matrices. The algorithms are implemented in \texttt{SageMath} and can be found in \cite{repo_sym2025}. In this place the reader can also find computed classification data for small IW and weighing matrices.\\

The paper is organized as follows. In section \ref{sec:NSOKS} we outline our version of the \texttt{NSOKS} algorithm for finding all ways to write an integer as a sum of $n$ squares. Our algorithm improves over an existing application \ref{alg:nsoks}. In section \ref{sec:3} we derive our algorithm for classification of the $IW$. Our approach constructs each representative matrix row by row, making sure that we do not hit twice the same class. To this end we define a complete ordering on rectangular integer matrices, called the \emph{row-lex} ordering, aiming at producing a minimum element in each class. This ordering has the property that every prefix of a minimum matrix in a class is itself minimum. For efficiency reasons, we relax the minimality condition for prefixes at some point, at the cost of producing multiple representatives. This problem is settled later by pruning the isomorphic matrices using the techniques of the next section.\\

In section \ref{sec:aut} we study the isomorphism problem and compute automorphism groups of IW-matrices. The computations are based on the graph isomorphism problem. Our implementation uses the graph theory library in \texttt{Sagemath}. For the automorphism groups we compute a set of generators and an identifier according to the taxonomy given in the abstract finite group database in LMFDB \cite{lmfdb}. As a supplement, we supply and implement an algorithm for verifying that we have actually computed the full automorphism group, and also proves (non-)isomorphism.\\

To avoid huge automorphism groups, we start by restricting the discussion to \emph{primitive} IW-matrices, which are those that are not diagonal block sums of smaller IW matrices. We then reconstruct the non-primitive case with the aid of Theorems \ref{thm:prim} and \ref{thm:wreath}.

Section \ref{sec:5} is devoted to the study and classification of symmetric and anti-symmetric IW matrices. We develop a method of finding all symmetric and anti-symmetric members in a class. We use this method to solve some open cases of symmetric and anti-symmetric $W(n,k)$. A necessary condition for the existence of such objects in a given class is that the matrix is H-equivalent to its transpose. Under such circumstances, we prove a group theoretic criterion (Proposition \ref{prop:symm}) for the existence of a symmetric object. Furthermore, the classification of the symmetric objects in a class is in bijection with a cohomology set (Proposition \ref{prop:hom}). As a consequence we prove that if the automorphism group is abelian, the number of symmetric subclasses in the class is a power of $2$ (Corollary \ref{cor:hom}). Our symmetric classification algorithm performs best in the primitive case. To this end we develop a theory of (anti-)symmetric classification from primitive symmetric and ordinary classification, as given in Theorems \ref{thm:symm} and \ref{thm:a-symm}.
As a case study, in \S\ref{sec:proj} we compute explicitly the symmetric classification for the projective space incidence and weighing matrices.\\

Section \ref{sec:6} yields the tools to compute the cardinality of $IW(n,k)$ as well as the cardinality of the (anti-)symmetric space, based on the classification of the primitive cases. The counting formulas are given in terms of the automorphism groups and generating functions. Propositions \ref{prop:countIW} and \ref{prop:count_sym_IW} give us effective tools to count the size of these spaces. Finally, in section \ref{sec:7} we apply all of the above practices to $IW(7,25)$ and tabulate detailed information on the primitive and full classification, automorphism groups, symmetric classification and cardinalities.\\

As a final note, we cannot escape the impression that symmetric, and to a lesser degree anti-symmetric IW and W matrices are abundant. Not only for the small sizes we are able to classify, but for the moderate sizes as well, e.g. sizes $23$ and $28$ which we have looked at. It was not too difficult to find the symmetric objects using our methods. We therefore propose that efforts to solve open cases of weighing matrices should begin by restricting to symmetric instances.

\section{The \texttt{NSOKS} algorithm}\label{sec:NSOKS}

The problem of representing an integer as a sum of squares goes back to Fermat, Lagrange and Gauss. Mathematicians like Hilbert,Landau,Hardy, Littelwood, Ramanujan, Davenport and more worked on this problem and the more general Waring problem. Jacobi theta functions and modular forms shed light on the problem of counting the number of representations of an integer as the sum of $r$ squares. Our modest contribution to this problem is the \texttt{NSOKS} algorithm. This algorithm computes the collection of all representations of a positive integer $n$ as a sum of $r$ nonnegative squares. The input is the number $n$, an integer $r$ for the number of required squares and an optional argument $maxsq$ which is an upper bound for the integers $s$ in the representation. The output is the full list $[S_1,\ldots,S_t]$, where each $S_i$ is a list $[(s_1,m_1),\ldots, (s_l,m_l)]$ with $0\le s_1<s_2<\ldots s_l\le maxsq$ such that $\sum m_i=r$ and $\sum m_is_i^2=n$. Each $S_i$ is a representation of $n$ as a sum of $r$ nonnegative squares. We may view the output as the list of all $PIW(r,n,1)$ up to H-equivalence.\\

A \texttt{Maple}\textregistered{} \ implementation already exists \cite{riel2006nsoks}. Nevertheless, our \texttt{Sagemath} implementation runs faster. For example, our \texttt{NSOKS}$(200,200)$ outputs $27482$ representations and runs on our machine in 0.3s. In comparison, the Maple code on the same machine runs in 13s. Both codes have been checked to give the same answers. Our algorithm advances by recursion, from the largest square down to zero. The algorithm loops on the largest square $s^2$ and its multiplicity $m_s$. Then a recursive call to \texttt{NSOKS} is taken, relative to the parameter change $n\to n-m_ss^2$, while setting $maxsq=s-1$. The main improvement over the maple code is the work with multiplicities, which saves a lot of time, and reduces the recursion depth. Also, once we get down to $maxsq=1$, we finish with no need for further recursion. See Algorithm \ref{alg:nsoks} for details.   

\begin{algorithm}
\caption{Find all representations of $n$ as a sum of $r$ nonnegative squares }\label{alg:nsoks}
\begin{algorithmic}[1]
    \Procedure{\texttt{NSOKS}}{$n,r,maxsq=False$}
    \If{$maxsq=1$} \textbf{return} $[[(1,n),(0,r-n)]]$ \Comment{{\scriptsize No need for recursion\label{step:1}.}}
    \EndIf
    \State $M\gets \lfloor \sqrt{n} \rfloor$
    \If{$maxsq$}
        \State $M\gets \min(maxsq,M)$
    \EndIf
    \State $L\gets \lceil \sqrt{n/r}\rceil$
    \State $SquaresList\gets []$
    \For{$s\in [L,M]$} \Comment{{\scriptsize Loop on the square.}}
        \For{$i\in [1,\lfloor n/s^2\rfloor]$} \Comment{{\scriptsize Loop on the multiplicity.}}
            \State $n'\gets n-i\cdot s^2$
            \If{$i=r$}
                \State Append $[(s,r),]$ to $SquaresList$.
            \Else 
                \State $rem\gets$ \texttt{NSOKS}$(n',r-i,maxsq=s-1)$ \Comment{{\scriptsize The recursion step.}}
                \For{$SubSquaresList \in rem$}
                    \State Append $[(s,i),*SubSquaresList]$ to $SquaresList$
                \EndFor
            \EndIf

        \EndFor
    \EndFor
    \State \textbf{return} SquaresList
    \EndProcedure
 \end{algorithmic}
\end{algorithm}
 \section{The row-lex ordering and the search algorithm}\label{sec:3}
 In this section we define the row-lex ordering on the set of integer matrices of a given size $m \times n$, and prove some basic properties. Using these properties, we design a search algorithm to find an exhaustive list of all $PIW(m,n,k)$ up to Hadamard equivalence. For efficiency reasons the output of the algorithm may still contain multiple candidates in a single Hadamard class. Using isomorphisms in next section as a post processing will fix this issue.\\
 
 \subsection{The row-lex ordering} The discussion here is not limited to partial weighing matrices, and we consider all integer matrices of a given size $p\times n$. We denote this set by $\ZZ^{p\times n}$. The set $\ZZ$ of all integers carries its natural ordering $\le$. We first extend this ordering to the set $\ZZ^n$ of $n$-vectors by the lexicographic extension of $\le$, still denoted $\le$. This means that
 \begin{equation}
    (v_1,\ldots,v_n)<(w_1,\ldots,w_n), \ \text{iff }\exists j  \ (v_1,\ldots,v_{j-1})=(w_1, \ldots, w_{j-1}) \text{ and } v_j<w_j.
\end{equation}
Next we extend this ordering to $m\times n$-matrices by lexicographic extension of $\le$ on the rows (i.e. viewing the matrix as a vector of rows). We write this ordering as $M\le_R N$. This is called the \emph{row-lex ordering}. Similarly we can consider the \emph{column-lex ordering}, by extending $\le$ on columns, viewing the matrix as a vector of columns. We shall write $M\le_C N$ for this ordering. The two orderings are not equal. As our algorithm builds the matrices row by row, it turns out that the row-lex ordering is more useful.\\

Let us introduce some notation.
{\rm For any matrix $M$, let $M_{i,j}$ be the $(i,j)$-entry, let $M_i$ denote its $i$th row and let $M^j$ denote its $j$th column. Let $M_{i:k}$ denote the submatrix whose rows are $M_i,M_{i+1},\ldots,M_{k-1}$ given in this order. We denote $M^{j:l}$ analogously for columns. } Let $(-1)_iM$ denote the matrix $M$ with $M_i$ replaced by $-M_i$. More generally, given a set of indices $S$, let $(-1)_SM$ be the matrix $M$ with the rows $M_i$ negated for each $i\in S$. Analogously we denote $(-1)^jM$ and $(-1)^SM$ for columns. For each matrix $M$ let $[M]$ denote its Hadamard equivalence class.

\begin{definition}  
\begin{equation}
Min(M) = \min \{A\ | \ A\in [M]\},
\end{equation} 
{\rm the minimum is taken with respect to the row-lex ordering. }
\end{definition}
{\rm We say that $M$ is \emph{minimum} if $M=Min(M)$. In each Hadamard class there exists a unique minimum matrix. We now study some properties of minimum matrices. We say that a vector $v$ \emph{begins with a positive (resp. negative) entry} if for some $j$, $v_1=\cdots=v_{j-1}=0$ and $v_j>0$ (resp. $v_j<0$).}

\begin{lemma}\label{lem:minmat}
    In a minimum matrix each nonzero row and each nonzero column begins with a negative entry.
\end{lemma}

\begin{proof} Let $M$ be minimum.
    Suppose that a row $M_i$ begins with a positive entry. Then $-M_i<M_i$, and by definition $(-1)_i M<M$, in contradiction to minimality.\\
    Now suppose that $M^j$ begins with a positive entry, sitting at the position $(i,j)$. Then in $(-1)^jM$, the first $i-1$ rows remain unchanged, while $((-1)^jM)_i<M_i$, which in turn implies that $(-1)^jM<M$, again contradicting the minimality of $M$.
\end{proof}
A nice consequence of this proof, not being used in this paper, is the following statement.

\begin{theorem}
    Each matrix $M$ can be brought, using only row and column negations, to a form where each nonzero row and column begins with a negative entry. 
\end{theorem}

\begin{proof}
    We modify $M$ to be the minimal matrix among all possible row and column negations. The proof of the Lemma \ref{lem:minmat} shows that now every nonzero row and column begins with a negative entry. 
\end{proof}

\begin{lemma}
    In a minimum matrix $M\in \ZZ^{m\times n}$, the \underline{columns} are in increasing order: $M^1\le M^2\le \cdots \le M^n$.
\end{lemma}

\begin{proof}
    Suppose by contradiction that $M^{j-1}>M^j$ for some $j$. Let $j$ be the smallest index with this property. Then for some $i$, $M_{s,j-1}=M_{s,j}$ for all $s<i$ and $M_{i,j-1}>M_{i,j}$. By swapping columns $j,j-1$ we obtain a matrix $M'$ in which rows $1,2,\ldots,i-1$ did not change, while row $i$ has decreased. Thus $M'<_R M$, a contradiction.
\end{proof}

The following is a key property in our algorithm.

\begin{lemma} \label{part_min}
    For the row-lex ordering, a matrix $M$ is minimum, if and only if for every $i$ the prefix $M_{1:i}$ is minimum.
\end{lemma}

\begin{proof}
    Clearly $M_{1:i}<_RM'_{1:i}$ implies $M<_RM'$. If $M_{1:i}$ is not minimal, then we can perform Hadamard operations on $M$ involving all columns and only the first $i$ rows, to decrease $M_{1:i}$. The resulting matrix $M'<_R M$, in contradiction to the minimality of $M$.
\end{proof}
We remark that in general the initial column submatrices $M^{1:j}$ of a minimum matrix $M$ need not be minimum in the column-lex ordering. This asymmetry is why we prefer to use the row-lex ordering when building the matrix row by row. 

\subsection{Minimizing a class}
In this short subsection we describe the sub-algorithm \texttt{MINCLASS} to find the minimum representative in a Hadamard class of a partial matrix.
Suppose that we are given a matrix $M\in \ZZ^{m\times n}$. Let $\Mon(m)$ denote the set of all monomial $m\times m$ matrices with values in $\{0,-1,1\}$. Let $Neg(M)$ denote the matrix obtained from $M$ by negating each column that begins with a positive entry. Let $Ord(M)$ be the matrix obtained from $M$ by permuting its columns to be written from left to right in increasing column order. Consider the following Algorithm \ref{alg:minclass}:

\begin{algorithm}
\caption{Minimizing a Hadamard class}\label{minclass}\label{alg:minclass}
\begin{algorithmic}[1]
    \Procedure{\texttt{MINCLASS}}{$M$}
    \State $m\gets $ height($M$), $n\gets$ width($M$)
    \State $Min\gets M$
        \For{$P\in \Mon(m)$}\Comment{go over all row negations and permutations}
        \State $N\gets PM$
        \State $N\gets Neg(N)$  
        \State $N\gets Ord(N)$ 
        \If{$N<_R Min$}
            \State $Min\gets N$
        \EndIf
        \EndFor
    \State \textbf{return} $Min$
    \EndProcedure
\end{algorithmic}
\end{algorithm}

\begin{lemma}
The procedure \texttt{\textsc{MINCLASS}}$(M)$ returns the minimal matrix in the class of $M$.  
\end{lemma}

\begin{proof}
    Let $M_0=PMQ^\top$ be the minimum matrix in the class $[M]$, $P,Q$ are monomial. The algorithm
 enumerates over $P\in \Mon(m)$ and for the correct $P$ we have $N:=PM=M_0Q$. It suffices to show that $M_0=Ord(Neg(N))$. The nonzero columns of both $Neg(N)$ and of $M_0$ all begin with a negative entry, so both matrices have the same multiset of columns, which means that $M_0=Neg(N)\Pi$ for a permutation matrix $\Pi$. Since the columns of $M_0$ are in increasing order, then necessarily $Ord(Neg(N))=M_0$.
\end{proof}

\begin{remark}
    Algorithm \ref{alg:minclass} performs exhaustive search on the row operations of $M$, and therefore its cost grows exponentially with the number of rows. 
\end{remark}

\subsection{The main search algorithm}
Now we turn to the main algorithm \texttt{RepPIW} (Algorithm \ref{alg:repiw}) which outputs a list of representatives of all Hadamard classes in $PIW(m,n,k)$. In its default implementation the program outputs exactly one matrix per class, the minimum one. However it contains an optional parameter, `mindepth', which greatly improves the running time, at the cost of getting multiple matrices per a single class. Before stating the algorithm we give a concise description.\\

The algorithm relies on Proposition \ref{part_min} that initial submatrices of a minimum matrix are minimum. 
The starting point is a list of all minimum integral vectors of weight $k$, which is in bijection with the output of \texttt{NSOKS}$(n,k)$. This gives the list for all minimum prefixes $PIW(1,n,k)$. At each stage the algorithm holds a list $MinPIW(p,n,k)$ of all minimum representatives of the $PIW(p,n,k)$. To each member $X\in MinPIW(p,n,k)$, 
we produce the list $LV(X)$ of all integral vectors of weight $k$ that are (i) bigger than the last row of $X$, and (ii) are orthogonal to all rows of $X$. Then for each $v\in LV(X)$ we obtain the matrix $X_{new}$ by stacking  $v$ below $X$. Using \texttt{MINCLASS}, we test if $X_{new}$ is minimum. We add it to the new list $MinPIW(p+1,n,k)$ iff it is minimum. Stopping at $p=m$, Proposition \ref{part_min} guarantees that we have correctly created a list of representatives for all Hadamard classes of $PIW(m,n,k)$.\\

One improvement that we add, which greatly affects the performance, is the parameter `mindepth' that tells the algorithm to stop using \texttt{MINCLASS} if $p>mindepth$. In this case we just stack any vector $v$ satisfying (i) and (ii), and add the new matrix to the list. The parameter `mindepth' is to be chosen so that there are not too many vectors left on the list, causing the output of the algorithm not to be too large. This saves a great deal of time on using \texttt{MINCLASS} which is time consuming. But this comes at the cost of obtaining multiple representatives in a class. Below in \S\ref{sec:aut} we will discuss the isomorphism problem, which will reduce the list to one representative per class. In the following algorithm we use the following notation: `SignedPerms$(v)$' returns the set of all permutations and element negations of a vector $v$. For a matrix $M$, recall that $M_i$ denotes its $i$th row. Let $M_-$ denote the matrix without its last row. Let $[M,v]$ denote the matrix $M$, augmented by the additional row $v$.\\

\begin{algorithm}
    \caption{Generating Hadamard representatives of $PIW(m,n,k)$}\label{RepPIW}\label{alg:repiw}
    \begin{algorithmic}[1]
    \Procedure{\texttt{RepPIW}}{$m,n,k,mindepth=m$}
    \State $SOKS\gets$ \texttt{NSOKS}$(n,k)$
    \State $MinPIW[1]\gets $ [\texttt{MINCLASS}$(v)$ for $v$ in $SOKS$] 
    \State $AllRows\gets$ SignedPerms$(SOKS)$ \Comment{{\scriptsize This is the full resrvoir of all possible rows.}} 
    \If{m=1}
        \State \textbf{return} $MinPIW[1]$
    \EndIf
    \For{$v \in MinPIW[1]$} \Comment{{\scriptsize Compute a list of all second rows.}}
            \State $R(v)\gets $[$w\in AllRows$ if $wv^\top=0$ \& $w>v$]
    \EndFor
    \For{$p=1$ to $m-1$}
        \For{$X\in MinPIW[p]$}
            \State $R(X)\gets $ [$w\in R(X_-)$ if $X_p w^\top=0$ \& $w>X_p$] \ \ \\ \Comment{Make a list of all $p$th rows. This tests conditions (i) and (ii)}
            \For{$w\in R(X)$}
                \State $X_{new}=[X,w]$
                \If{$p\le mindepth$}
                    \If{$X_{new}==$ \texttt{MINCLASS}$(X_{new})$}
                        \State Append $X_{new}$ to $MinPIW[p+1]$.
                    \EndIf
                \Else
                    \State Append $X_{new}$ to $MinPIW[p+1]$.
                \EndIf
            \EndFor
        \EndFor
    \EndFor
    \State \textbf{return} $MinPIW[m]$
    \EndProcedure
    \end{algorithmic}
\end{algorithm}

\begin{remark}
    Notice that in this algorithm we keep, together with the prefix $X$, also a list $R(X)$ of all vectors $v$ that are orthogonal to $X$ and greater than the bottom row of $X$. We filter $R(X)$ from $R(X_-)$ by checking that the vectors are orthogonal to the bottom row (line 13). Thus we only need to check orthogonality to a single vector, and just filter it from the smaller list $R(X_-)$.
\end{remark}

\begin{theorem}
    The function \textsc{\texttt{RepPIW}}$(m,n,k)$ outputs the list $MPIW(m,n,k)$ of all minimum Hadamard representatives of $PIW(m,n,k)$.\\ The function 
    \textsc{\texttt{RepPIW}}$(m,n,k,mindepth=d)$ outputs a larger list of $PIW(m,n,k)$ containing all minimum elements.
\end{theorem}

\begin{proof}
    The proof of the first part is by induction on $m$. For $m=1$ this is clear, as $MPIW(1,n,k)$ is the list of all minimum vectors is $SOKS$. Assuming validity for $m-1$, we enter the for loop at $p=m-1$ (line 11) with $MinPIW(m-1)=MPIW(m-1,n,k)$ by the induction hypothesis. Suppose that $M\in MPIW(m,n,k)$. Then by Theorem \ref{part_min} $M_{1:m-1}\in MinPIW(m-1)$. The list $R(M_{1:m-1})$ holds all vectors that are orthogonal to $M_{1:m-1}$. Thus the vector $M_m$ enters the list $R(M_{1:m-1})$ (line 13) and passes the minimality test (line 17), allowing $M$ to enter the list $MinPIW(m)$ (line 18). This proves that $MinPIW(m)\supseteq MPIW(m,n,k)$. The opposite inclusion is clear as line $18$ allows only minimum matrices. This proves the first assertion. The second assertion follows easily, as we do not always perform the minimality test, but yet the minimum matrices pass all tests. 
\end{proof}

\section{H-equivalences and Automorphisms} \label{sec:aut}

We can rephrase the notion of H-equivalence as follows. 

\begin{definition}
    Two $PIW(m,n,k)$ $A,B$ are \emph{H-equivalent} if there are monomial matrices over $\{0,-1,1\}$, $L\in \Mon(m)$ and $R\in \Mon(n)$ such that $B=LAR^\top$. If $A,B$ are square, then we say that $A,B$ are \emph{TH-equivalent} if $A$ is H-equivalent to $B$ or to $B^\top$. We also say that square $A,B$ are \emph{symmetric Hadamard equivalent (SH-equivalent)} if we can take $B=LAL^\top$, $L\in \Mon(m)$.
\end{definition}
These are easily seen to be equivalence relations on the space of all integer matrices, and they preserve the space of PIW. An H-equivalence is also called as \emph{isomorphism} of PIW. An \emph{automorphism} of a PIW $A$ is a self isomorphism, i.e. a pair $(L,R)$ of monomial matrices satisfying $A=LAR^\top$. The set of all automorphisms is closed under pairwise matrix multiplication and inversion, hence forms a group, called the \emph{automorphism group} and denoted by $\Aut(A)$. This group has a subgroup consisting of only unsigned permutations. We denote this subgroup by $\PermAut(A)$.

In this section we show how to compute isomorphisms and automorphisms of PIWs. The primitive we shall use is the graph isomorphism problem. The method we present here is not new, and is used by many authors, see e.g.\cite{O_Cathain2011-wh}. For convenience we chose to present it in detail here. Let $A\in \ZZ^{m\times n}$ be any integer matrix. Let $I_2$ denote the $2\times 2$ identity matrix, $J_2$ the $2\times 2$ all 1's matrix. For any integer $m\in \ZZ$ we define 
$$e(m)\ := \ \begin{cases}
    mI_2 & m\ge 0\\
    m(I_2-J_2) & m<0
\end{cases}.$$

The \emph{extended matrix} $E(A)$ is the $2m\times 2n$ matrix given by $E(A)=(e(A_{i,j}))_{i,j}$. It satisfies $E(A^\top)=E(A)^\top$, $E(LA)=E(L)E(A)$ and $E(AR)=E(A)E(R)$ for monomial matrices $L,R$. $E(-)$ does not respect matrix multiplication in general. Notice that for monomial matrices $M$, $E(M)$ is a permutation matrix. We have the following easy lemma.

\begin{lemma}\label{lem:EAut}
    For an integer matrix $A$, the map $(L,R)\mapsto (E(L),E(R))$ defines an injection $E_*:\Aut(A)\hookrightarrow \PermAut(E(A))$.
\end{lemma}

\begin{proof}
    This map is a homomorphism as $E$ respects multiplication by monomial matrices. The map $e:\ZZ\to M_2(\ZZ)$ is clearly injective, and so is the map $E$.
\end{proof}

We can also describe the image of $E_*$.
In the following we shall write $I$ for the identity matrix without specifying the size, which should be clear from the context. The matrix $E(L)$ clearly commutes with the matrix $E(-I)$, since $L$ commutes with $-I$. We prove

\begin{lemma}\label{lem:Ecomm}
    A pair $(P,Q)\in \PermAut(E(A))$ is in the image of $E_*$, if and only if both $P,Q$ commute with $E(-I)$.
\end{lemma}

\begin{proof}
    The `only if' part is clear, so we only prove the `if' part. Suppose that $P,Q$ both commute with $E(-I)$. We will first show that $P=E(L)$ for a suitable monomial $L$. Let the size of $P$ be $2n$. The matrix $E(-I)$ is a permutation matrix of the involution interchanging $2i$ with $2i+1$ for all $0\le i\le n-1$. By the commutation assumption, $P$ corresponds to a permutation on the set $\{1,2,\ldots,2n\}$, taking each unordered pair $\{2i,2i+1\}$ to some unordered pair $\{2j,2j+1\}$. The correspondence $i\to j$ defines a permutation $\pi$, and we let $L_{\pi(i),i}=L_{j,i}=1$ if $P$ takes $2i$ to $2j$, $L_{j,i}=-1$ if $P$ takes $2i$ to $2j+1$, and $L_{j,i}=0$ otherwise. This defines a monomial matrix $L$ and from the definition $E(L)=P$. Similarly there is a monomial $R$ such that $E(R)=Q$. At this point we know that $(P,Q)=(E(L),E(R))$ for suitable monomials $L,R$, and moreover, $L$ and $R$ are unique. It remains to show that $(L,R)$ is an automorphism. We know that $E(LAR^\top)=E(L)E(A)E(R)^\top=PE(A)Q^\top=E(A)$. By the injectivity of $E$, $LAR^\top=A$.
\end{proof}

The proof gives us a way to compute $\Aut(A)$. We compute $\PermAut(E(A))$ and $\Aut(A)$ is isomorphic to the centralizer of $(E(-I),E(-I))$. Each centralizing element corresponds to an automorphism of $A$ by the recipe given in the proof.

\subsection{Computing $\PermAut(E(A))$}\label{subsec:PermAut}
We now reduce the problem to computation of automorphism of graphs. This is done by treating $E(A)$ as an adjacency matrix of a directed weighted bipartite graph. If $A$ has size $m\times n$, we define the adjacency matrix $\mathcal B(A)$ as the matrix 
$$ \mathcal B(A)\ := \  \begin{bmatrix}
    0 & E(A)\\ 0 & 0
\end{bmatrix}.$$ This is a matrix of size $(2m+2n)\times (2m+2n)$ defining a directed weighted bipartite graph $\mathcal G(A)$ on a vertex set $X\sqcup Y$ with $|X|=2m$, $|Y|=2n$ and edges going out from $X$ and in to $Y$, with weights specified by the matrix. The automorphisms of $\mathcal G(A)$ are permutation matrices $\Pi$, satisfying $\Pi(\mathcal B(A))\Pi^\top=\mathcal B(A)$. Each automorphism $\Pi$ is a block sum $\Pi=P\oplus Q$, and satisfies $PE(A)Q^\top=E(A)$, conversely, every pair $(P,Q)\in \PermAut(E(A))$ defines an automorphism $\Pi=P\oplus Q$. Thus we have an isomorphism $$\Aut(\mathcal G(A))\cong \PermAut(E(A)),$$ given by this recipe.\\

For the computation of $\Aut(\mathcal G(A))$ we use the \texttt{Sagemath} implementation .automorphism\_group() in the graph theory library.

\subsection{Computing equivalences.} Let $A,B$ be two integer matrices, which we want to test for H-equivalence. To find the equivalence if any, we compute the extended matrices $E(A)$ and $E(B)$ and the directed bipartite graphs $\mathcal G(A)$ and $\mathcal G(B)$. Let us denote $\Isom(A,B)$ the set of H-equivalences from $A$ to $B$, $\PermIsom(R,S)$ the set of permutation two sided equivalences from $R$ to $S$, and $\Isom(\mathcal G,\mathcal H)$ the set of graph isomorphisms between the two weighted directed graphs $\mathcal G$ and $\mathcal H$. Like with automorphisms, there are an injection and an isomorphism $\Isom(A,B)\hookrightarrow \PermIsom(E(A),E(B))\cong \Isom(\mathcal G(A),\mathcal G(B))$. The image of the injection is the set of all pairs that commute with $(E(-I),E(-I))$. The proofs are similar to the ones of Lemmas \ref{lem:EAut}, \ref{lem:Ecomm} and \S \ref{subsec:PermAut}, and we do not repeat them here. In our implementation we use the \texttt{Sagemath} command .is\_isomorphic() which can also supply the isomorphism. From this isomorphism we track back to the isomorphism between $A$ and $B$, as explained above.

\subsection{Primitive matrices}
For the rest of the paper, we restrict the discussion to square IW-matrices.
It could happen, even for small matrices, that the automorphism group will be huge, which may be a hurdle for some computational tasks. To give an extreme example, the automorphism group of $A=kI_n$ is the collection of all pairs $(L,L)$, where $L$ ranges over the entire set $\Mon(n)$. It is advisable to restrict attention first to primitive matrices. The general case is taken care of by Theorem \ref{thm:wreath} below. In the following we shall write $A\oplus B$ for the block diagonal sum of the square blocks $A$ and $B$, and  similarly for longer block sums. Let $A^{\oplus r}$ denote the block diagonal sum of $r$ copies of $A$.

\begin{definition}
    An IW-matrix $A$ is said to be \emph{primitive}, if it is not H-equivalent to a diagonal block sum of smaller blocks. Otherwise we say that $A$ is \emph{imprimitive}.
\end{definition}
By induction any matrix is up to H-equivalence a direct sum primitive matrices. For any integer matrix $A$, let $\chi(A)$ be its characteristic matrix, which is the matrix of the same size defined by $\chi(A)_{i,j}=1$ if $A_{i,j}\neq 0$ and $\chi(A)_{i,j}=0$ otherwise. Let $\mathcal B_A$ be the (directed unweighted) bipartite graph defined by $\chi(A)$. An H-equivalence between $A$ and $B$ induces an isomorphism between $\mathcal B_A$ and $\mathcal B_B$.  We have

\begin{lemma}
    Let $A$ be an IW. Then $A$ is primitive if and only if the graph $\mathcal B_A$ is connected.
\end{lemma}

\begin{proof}
    Let $A=IW(n,k)$. Suppose first that $A$ is primitive. By way of contradiction, after permuting the rows and columns of $A$, we assume that the vertex sets $\{0,1,\ldots c-1 \}$ and  $\{0,1,\ldots,d-1\}$ form the two sides of a connected component of the graph for $c<n$. The remaining matrix entries are supported in $\{c,\ldots,n\}\times \{d,\ldots,n\}$. Since the matrix is nonsingular, we must have that $c=d$, and the matrix $A$ can be permuted on rows and columns to be a block diagonal sum of smaller blocks, a contradiction. In the opposite direction, if $A$ is non-primitive, then  $A\sim_H A_1\oplus A_2$ for smaller matrices, hence $\mathcal B_A\cong \mathcal B_{A_1\oplus A_2}$, and the graph is disconnected.
\end{proof}

It follows from the lemma, that if $A_i$ are primitive $IW$, then $\mathcal B_{A_i}$ are the connected components of $\mathcal B_{A_1\oplus\cdots\oplus A_r}$. As a consequence we can prove 

\begin{theorem}[Primitive Decomposition]\label{thm:prim}
    Let $A$ be an IW matrix. Then up to H-equivalence $A$ is a direct block sum of primitive $IW$s. The decomposition is unique in the following sense. If $A_1\oplus \cdots \oplus A_r \sim_H A'_1\oplus\cdots\oplus A'_s$ for primitive $A_i,A'_i$, then $r=s$ and there is a permutation $\pi\in S_r$ such that for all $i$, $A_i\sim_H A'_{\pi(i)}$.
\end{theorem}

For the proof we set up some notation. We write a monomial matrix $M\in \Mon(n)$ uniquely as a product $M=SP$, where $P$ is a permutation matrix and $S$ is a diagonal $\{\pm 1\}$-matrix. We write $P=|M|$. If $(L,R)$ is an H-equivalence from $A$ to $B$, then $(|L|,|R|)$ is the isomorphism between $\chi(A),\chi(B)$ as adjacency matrices of bipartite graphs.

\begin{proof}
    First we prove primitive decomposition. If $A$ is primitive then it is its own primitive decomposition. Or else, $A\sim_H A_1\oplus A_2$ for smaller size matrices. By induction on the size $A_1,A_2$ have primitive decompositions, and we finish since $B_1\sim_H C_1$ and $B_2\sim_H C_2$ implies $B_1\oplus B_2\sim_H C_1\oplus C_2$.\\

    To prove uniqueness, notice first that $r=s$ is the number of connected components of $\mathcal B_{A_1\oplus\cdots\oplus A_r}\cong \mathcal B_{A'_1\oplus\cdots\oplus A'_s}$, and the $\mathcal B_{A_i}$ and separately $\mathcal B_{A'_i}$ correspond to the connected components of each graph. Let $(L,R)$ be the H-equivalence between the block sums. Let $P=|L|$ and $Q=|R|$ be the permutation matrices subordinate to $L,R$. Then $(P,Q)$ gives the isomorphism between the total bipartite graphs, and must take connected components to connected components. Thus for every $i$, $(P,Q)$ takes the block $\chi(A_i)$ to another block $\chi(A'_{\pi(i)})$ for some function $\pi$. The function $\pi$ is a bijective matching between the sets of connected components of both graphs, hence it is a permutation. Now the H-equivalence $(L,R)$ must take $A_i$ to $A'_{\pi(i)}$ and induce an isomorphism between them. This completes the proof.
\end{proof}

\begin{corollary}
    Two IW matrices are H-equivalent if and only if the multisets of H-equivalence classes of their primitive components are equal.$\Box$
\end{corollary}

The primitive decomposition allows us to compute the automorphism group of an IW matrix $A$, given the primitive decomposition of $A$ and the automorphism group of each primitive component. It suffices to assume that $A=A_1\oplus\cdots\oplus A_r$, for primitive $A_i$. In the following, let $H\wr S_r$ denote the wreath product of $S_r$ on $H^r$, which is the semidirect product $H^r\rtimes S_r$ where $S_r$ acts on $H^r$ by permuting coordinates. Another well-known description of the wreath product is given in terms of group monomial matrices. Let $H$ be any group. An $H$-monomial $n\times n$ matrix is a $\{0\}\cup H$-valued $r\times r$ matrix, with one nonzero entry in each row and column. We denote this set by $\Mon(H,r)$. It is closed under matrix multiplication and isomorphic to $H\wr S_r$. We encode each element of $H^r\wr S_r$ as a pair $((h_i)_i,\pi)$, where $h=(h_i)\in H^r$ and $\pi\in S_r$ is a permutation. This pair corresponds to the $H$-monomial matrix whose $(\pi(i),i)$ coordinate is $h_i$.

\begin{theorem} \label{thm:wreath}
    Let an IW 
    $$A\ \sim_H\ A_1^{\oplus r_1} \oplus A_2^{\oplus r_2}\oplus \cdots \oplus A_s^{\oplus r_s}.
    $$ 
    for primitive $A_i$, such that for $i\neq j$, $A_i$ and $A_j$ are not H-equivalent. Then 
    \be \Aut(A)\ \cong \ \prod_{t=1}^s \Aut(A_t)\wr S_{r_t}.
    \ee 
    In particular $|\Aut(A)|=\prod_{t=1}^s |\Aut(A_t)|^{r_t}r_t!$.
\end{theorem}
For the proof we use the notation $A[I]$ to mean the rectangular submatrix of $A$, whose row index set is $I$, and $A[I\times J]$ for the submatrix with entries in $(i,j)\in I\times J$. $I$ and $J$ are taken with the induced ordering from the integers.
\begin{proof}
    Since the isomorphism type of $\Aut(A)$ depends only on the class $[A]$, we may assume that $A=A_1^{\oplus r_1}\oplus A_2^{\oplus r_2}\oplus\cdots\oplus A_s^{\oplus r_s}$. From Theorem \ref{thm:prim} and its proof we know that any automorphism of $A$ must map each copy of $A_i$ to some copy of $A_i$. In more precise terms, suppose that $I_i^j\times I_i^j$ is index set of $j$th block of each $A_i$. If $(L,R)$ is an automorphism, then for each $i,j$ there exists an index $k=k(i,j)$ such that $L[I_i^k]$ is supported on $I_i^k\times I_i^j$, $R[I_i^k]$ is supported on the same $I_i^k\times I_i^j$, $L[I_i^k\times I_i^j]$ and $R[I_i^k\times I_i^j]$ are hence monomial, and $(L[I_i^k\times I_i^j],R[I_i^k\times I_i^j])$ is an automorphism of $A_i$. This gives a mapping $\nu:\Aut(A)\to \prod_{i=1}^t \Mon(\Aut(A_i),r_i)$, where for each $i$ the underlying permutation $\pi_i\in S_{r_i}$ satisfies $\pi_i(j)=k(i,j)$,  and the $\Aut(A_i)$-monomial matrix is $((L( I^{\pi_i(j)}_i\times I_i^j),R(I^{\pi_i(j)}_i\times I_i^j)_j,\pi_i)$. Since matrix multiplication respects blocks, this mapping is a homomorphism. We can also construct the inverse map. Given a data $(((L_{i,j},R_{i,j})_i,\pi_i))$, where $(L_{i,j},R_{i,j})\in \Aut(A_i)$, we construct a monomial pair $(L,R)$ supported on $\bigcup_{i,j} I_i^j\times I_i^{\pi_i(j)}$, where $L$ is the block sum of $L_i := ((L_{i,j})_j,\pi_i)$ and $R$ is a diagonal block sum or $R_i\ := ((L_{i,j})_j,\pi_i)$. The pairs $(L_i,R_i)$ are each an automorphism of $A^{\oplus r_i}$, and $(L,R)\in \Aut(A)$ is a $v$-preimage of the given data. 
\end{proof}

\subsection{Verifying the automorphism group and isomorphisms}
Our way of computing isomorphisms and automorphism groups relies on the graph theory library of \texttt{Sagemath}. We would like to present here a complementary approach, which supplies independent proofs that two IW-matrices are (non-)isomorphic, and a certification that we have indeed computed the full automorphism group. We feel that our paper will be more complete if we are doing this practice. This addition is not intended to replace the previous approach. It is meant to be supplement, and in fact it is most efficient when applied as an augmentation to the graph theory output.  Let us treat first the isomorphism problem.

If our graph theory based program returns an isomorphism between $A$ and $B$, then we can readily check it, and there is no problem there. However, how do we add an independent proof to a non-isomorphism? To this end we introduce the \emph{Code Invariant}. For an integral matrix $D\in [-L,L]^{d\times n}$ let $Code(D)=Code_L(D)=\mathbf{b}D$, where $\mathbf{b}=[b^{d-1},\ldots,b^2,b,1]$, $b=2L+1$. The correspondence $D\mapsto Code(D)$ is an injection on $[-L,L]$ integer matrices. We write $D\prec_{d}M$ if $D$ is a $d\times n$ submatrix of $M$. For any matrix $M\in [-L,L]^{r\times n}$, we define the \emph{code invariant} to be 
$$ CodeInv(M,d)\ := \ \text{Multiset}\{Code(Min(D))\ | \ D\prec_d M   \}.
$$
This is clearly an H-equivalence invariant of $M$. When $CodeInv(A,d)\neq CodeInv(B,d)$, this is a proof that $A$ is not isomorphic to $B$. Practically, for the small scale matrices we consider in this paper, $CodeInv(-,4)$ is almost always sufficient.\\

Let us turn now to the more difficult question of validating the automorphism group. Our graph-theory based program returns a subgroup $G=\{(L_i,R_i)\}\le \Aut(A)$ of automorphisms, for which we can readily check that they are indeed automorphisms. 
The question is: Is $G$ the full group of automorphisms? We want to have an algorithm that decides on this question, and further if not, output at least one (new) element in $\Aut(A)\smallsetminus G$. Our algorithm will be based on the following primitive: Let $X$ be a finite set on which $\Aut(A)$ acts. The following lemma is trivial.
\begin{lemma}
    Suppose that $X=X_1\cup X_2\cup \cdots\cup X_r$ is the orbital decomposition of $X$ with respect to the action of $G$. Choose elements $x_i\in X_i$. Then $\Aut(A)\neq G$ if one of the two incidents occur:
    \begin{itemize}
        \item[(i)] There exists an element $\sigma\in \Aut(A)$ such that $\sigma x_1=x_j$ for $j>1$;
        \item[(ii)] There exist an element $\sigma\in \Aut(A)\smallsetminus G$ such that $\sigma x_1=x_1$.$\Box$ 
    \end{itemize}
\end{lemma}

Let $\nn:=\{0,1,\ldots,n-1\}$ and let $\Perm(n)$ denote the group of permutation matrices on $\nn$. There is a group homomorphism $\Aut(A)\to \Perm(n)$ given by $(L,R)\mapsto |L|$. This homomorphism need not be injective, but the map $(L,R)\mapsto L$ is an injection $\Aut(A)\to \Mon(n)$, since $R$ can be recovered from $L$ by the nonsingularity of $A$. Let $\nn^{(r)}$ be the set of all distinct $r$-tuples taken from $\nn$. As our working set we will take $X=X_r:=\nn^{(r)}$, on which $\Aut(A)$ acts via the action of $|L|$ on $\nn$. The parameter $r$ can be chosen to optimize the algorithm. Suppose that we are given a set of generators $\{g_i\}$ of $G$. To find the $G$-orbits in $X$, we construct a graph $\Gamma(\{g_i\},X)$, whose vertex set is $X$, and $(a,b)\in X^2$ is an edge, if and only if $b=g_ia$ for some $i$. The $G$-orbits are the connected components of the graph. The following algorithm outputs all $\sigma\in \Aut(A)$ satisfying $\sigma x_1=x_j$.  Let $\Sigma_r$ be the set of all diagonal sign matrices of size $r$. If $S\in \Sigma_r$ and $I\subseteq \langle n\rangle$ is a subset with $|I|=r$, let $S_!\in \Sigma_n$ be the sign diagonal matrix with $S_{!,k_i,k_i}=S_{i,i}$ for $I={k_1\le \cdots \le j_r}$ and $S_{!,j,j}=1$ otherwise. Recall that $A[I]$ is the submatrix of $A$ relative to the row index set $I$. We say that matrices $A,B$ are row (resp. column) equivalent if $A=LB$ for a monomial $L$ (resp. $A=BR^\top$ for a monomial $R$). It is easy to check row equivalence. We just normalize the nonzero rows of both $A,B$ to begin with a negative entry and then compare the multisets of rows.\\

\begin{algorithm}
    \caption{Finding $\sigma\in \Aut(A)\smallsetminus G$ satisfying $\sigma x_1=x_j$.}\label{alg:newaut}
    \begin{algorithmic}[1]
        \Procedure{\texttt{NewAut}}{$X,G,r$}
            \For{a basepoint $x_j$ in orbit \#$j$ of $X_r$}
                \State $B_1\gets A[x_1]$ and $B_j\gets A[x_j]$
                
                \For{$S\in \Sigma_r$}
                    \State $B_1^{ord}\gets Ord(Neg(SB_1))$ and $B_j^{ord}\gets Ord(Neg(B_j))$
                    \If{$B_1^{ord}\neq B_j^{ord}$}
                        \State Continue \Comment{ $SB_1$ and $B_j$ are not column equivalent}
                    \EndIf
                    \State Set $\nn=J_1 \cup\cdots\cup J_r$ a partition by equal columns of $B_j^{ord}$. 
                    \State Let $R_1^{ord}\in \Mon(n)$ satisfy $SB_1R_{1}^{ord}=B_1^{ord}$.
                    \State Let $R_j^{ord}\in \Mon(n)$ satisfy $B_jR_{1}^{ord}=B_j^{ord}=B_1^{ord}$.
                    \For{$k=1$ to $r$}
                        \If{$B_1^{ord}(J_k)=0$}
                            \State $\Gamma_k\gets Mon(J_k)$
                        \Else 
                            \State $\Gamma_k\gets \Perm(J_k)$
                        \EndIf
                    \EndFor
                    \State \Comment{$R_1^{ord}$ is not unique. Can be modified by elements of $\prod_k \Gamma_k$}
                    \State \Comment{The next block prunes $\prod_k \Gamma_k$}
                    \For{$k=1$ to $r$}
                        \State $B_j^{k,ord}\gets (AR_j^{ord})[\nn\smallsetminus J_k\times x_j]$
                        \State $B_1^{k,ord}\gets (AR_1^{ord})[\nn\smallsetminus J_k\times x_j]$
                        \State $Prune_k\gets \{ \}$
                        \For{$\gamma_k\in \Gamma_k$}
                            \If{$B_j^{k,ord}\gamma_k^\top$ is row equivalent to $B_1^{k,ord}$}
                                \State Add $\gamma_k$ to $Prune_k$
                            \EndIf
                        \EndFor
                    \EndFor
                    \State \Comment{We enumerate on $\prod_k Prune_k\subset \Mon(n)$}
                    \For{$Q\in \prod_k Prune_k$}
                        \If{$S_!AR_1^{ord}$ is row equivalent to $AR_j^{ord}Q^\top$}
                            \State Let $P\in \Mon(n)$ s.t. $PS_!AR_1^{ord}=AR_j^{ord}Q^\top$
                            \State $(L,R)\gets (PS_!,R_j^{ord}(R_i^{ord}Q)^\top)$
                            \If{$(L,R)\notin G$} \Comment{If $j\neq 1$ this will always be the case; no need to check. If $j=1$, only need to check that it is not in the point stabilizer of $x_1$}
                                \State Return (L,R) \Comment{This is a new automorphism}
                            \EndIf
                        \EndIf
                    \EndFor
                \EndFor
                \State Return True \Comment{No new automorphisms. $G=\Aut(A)$}
            \EndFor
        \EndProcedure
    \end{algorithmic}
\end{algorithm}

Let us explain how algorithm \ref{alg:newaut} works and prove it. Enumerating over a basepoint $x_j$ in each orbit of $X_r$ (line 2), suppose that $(L,R)$ is an automorphism of our IW matrix $A$, with $L$ taking $x_1$ to $x_j$ and $(L,R)\notin G$. In particular it induces an isomorphism $(S,R)$ between $B_1:=A[x_1]$ and $B_j:=A[x_j]$. Moreover $S$ is only a signature matrix, since $x_1,x_j$ are ordered tuples. In line $4$ we enumerate over $S$. Now $SB_1$ and $B_j$ must be column equivalent, which is the same as to say that $B_1^{ord}:=Ord(Neg(B_1))=B_j^{ord}:=Ord(Neg(B_j))$. If this is not the case (line 6), we move on to the next $S$. We let $J_1\cup J_2\cup\cdots\cup J_r$ be the partition of $\nn$ according to the equal columns in $B_j^{ord}=B_1^{ord}$. Now is the time to find $R$. The issue is that $R$ need not be because of equal columns. In lines 9-10 we define $R_1^{ord}$ and $R_j^{ord}$ that bring $SB_1$ and $B_j$ respectively to $B_1^{ord}$. Our desired $R$ is now $R_1^{ord}QR_j^{ord\top}$ for $Q\in \prod_k \Gamma_k$, where $\Gamma_k$ is the permutation group (or the monomial group for zero-columns) (lines 12-18). This product group may be too large, and enumerating on $Q$ can thus be too heavy. Lines 20-29 prune this group down to a much more manageable size. The idea is treat each $J_k$ separately and enumerate on $\gamma_k\in \Gamma_k$. Then we test if the residual matrices at $(\langle n\rangle \setminus x_1)\times J_k$ and $(\langle n\rangle \setminus x_j)\times J_k$ are row equivalent w.r.t. our choice of $\gamma_k$ (line 26). The $\gamma_k$ that survive define the set $Prune_k$. This reduces the search to the smaller product $\prod_k Prune_k$. Enumerating on elements $Q$ in this group yield the candidates for $R$, and for each candidate we check (line 33) that the two matrices are row equivalent (which should now reveal $L$, line 34). If this is the case, we have an automorphism $(L,R)$ and we finally check whether it belongs to $G$. If $j\neq 1$, this will automatically be true, since $x_j$ is not in the same $G$-orbit of $x_1$. For $j=1$, we only needs to check that $(L,R)$ is not in the point stabilizer of $x_1$. Otherwise, we proceed to the next candidate $S$ and $Q$. If no automorphism outside $G$ was found, we declare that $\Aut(A)=G$.\\

Our use of this algorithm is as follows. It either proves that $\Aut(A)=G$, or finds a new element in $\Aut(A)\smallsetminus G$. In the latter case we abort the algorithm, add our new element to $G$, generate a new group (renaming it $G$) and feed it back to Algorithm \ref{alg:newaut}. This algorithm may also be used to compute $\Aut(A)$ in the first place, without appealing to the graph theory method, but it is slower, so we only use it to validate that the graph theory program has computed the correct group.

The complexity of this validation algorithm is largely affected by the choice of $r$ in $X=\nn^{(r)}$. If $r$ is too small, then we could have many zero-columns in $B_1^{ord}$, which will make $Prune_k$ big. If $r$ is too large, then the enumeration on $S$ will take more time. The choices of $x_1$ and $x_j$ could matter too.

\section{Classifying Symmetric and Antisymmetric IW}\label{sec:5}
We can use the automorphism group to find symmetric and anti-symmetric IW matrices. Moreover, we are able to classify all (anti-)symmetric $IW(n,k)$ up to \emph{symmetric Hadamard equivalence}. We denote an (anti-)symmetric (integer) weighing matrix by $SIW(n,k)$,$AIW(n,k)$, $SW(n,k)$ and $AW(n,k)$ respectively.

\begin{definition}
        Two (anti-)symmetric IW matrices $A,B$ are \emph{symmetric Hadamard equivalent} (SH-equivalent) in short, if there exists a monomial matrix $M$ such that $B=MAM^\top$.
        The set of all $M$ such that $A=MAM^\top$ is a group under matrix multiplication, and is called the \emph{symmetric automorphism group} of $A$. We denote this group by $\SAut(A)$   
\end{definition}

Let us first discuss an easier task, of finding all (anti-)symmetric $IW(n,k)$ up to H-equivalence. This means that we need to find an (anti-)symmetric representative in each H-equivalence class, if one exists. We say that an H-equivalence class $[A]$ is \emph{symmetric} if $A\sim_H A^\top$. This is easily seen to be a property of the class, and independent of the chosen representative $A$. A necessary condition that $[A]$ will contain an (anti-)symmetric matrix is that it is a symmetric class. This condition turns out to be not sufficient, but surprisingly for the small scale classes we considered, in most cases a symmetric class contained a symmetric matrix. Anti symmetric matrices occur less often, and there are cases where a class contains both a symmetric and an anti-symmetric member. 

Suppose now that $[A]$ is a symmetric class, and we would like to find an (anti-) symmetric $S\in [A]$. Let $(L,R)$ be an isomorphism from $A$ to $A^\top$, that is, $A^\top=LAR^\top$. 

\begin{proposition}\label{prop:symm}
    If $A\in IW(n,k)$ satisfies $A^\top=LAR^\top$ for monomial $(L,R)$. Then $[A]$ contains a symmetric matrix, if and only if there exists a monomial $M$ such that $(M^\top L,MR)\in \Aut(A)$, in which case $MA\in [A]$ is symmetric. Likewise, $[A]$ contains an anti-symmetric matrix, if and only if there exists a monomial $M$ such that $(-M^\top L,MR)\in \Aut(A)$, in which case $MA\in [A]$ is anti-symmetric.
\end{proposition}

This proposition reduces the search for (anti-)symmetric representatives, to a search over the automorphism group. The following corollary states the existence of previously unknown $SW$ matrices.

\begin{corollary}
    There exist symmetric weighing matrices $SW(23,16),SW(30,17)$,\\
    $SW(28,25)$ and an anti-symmetric weighing matrix $AW(28,25)$.
\end{corollary}
\begin{proof}
The matrix $SW(23,16)$ appears in \cite{munemasa2017weighing}. The other matrices appear in \cite{repo_sym2025}.
 All of these matrices were obtained using the above proposition on a known initially non symmetric weighing matrix.  
\end{proof}

We comment that recently $SW(30,25)$ \cite{Georgiou2023-ks}, $SW(22,16)$ \cite{Topno2025-bp}, $SW(19,9)$,$SW(18,9)$ and $AW(18,9)$ \cite{10.1007/978-3-319-17729-8_19} were found.

\begin{proof}[Proof of Proposition \ref{prop:symm}]
    Let us prove the symmetric part. The anti-symmetric part is similar. Suppose first that such a monomial $M$ exists. Then 
    \begin{multline*}
        (MA)^\top=A^\top M^\top=(LAR^\top)M^\top=LA(MR)^\top\\
        =L(M^\top L)^\top (M^\top L)A(MR)^\top =L(M^\top L)^\top A=MA,
    \end{multline*}
    so $MA\in SIW(n,k)$. Conversely, if $MA$ is symmetric, then 
    $$MA= A^\top M^\top=LAR^\top M^\top=LA(MR)^\top \implies A=(M^\top L)A(MR)^\top.$$
\end{proof}

Proposition \ref{prop:symm} also has a group theoretic facet. If $[A]$ is a symmetric class, we can define a larger automorphism group of $A$, denoted by $\TAut(A)$, which is defined as follows: As a set, $\TAut(A)$ contains all pairs $(L,R)\in \Aut(A)$, and symbols $(L',R')\top$ for each monomial $L',R'$ such that $A=L'A^\top R^{'\top}$. This notation bears the meaning that we first transpose $A$ and then apply $(L',R')$. The multiplication rules are $(L_1,R_1)\top \cdot (L_2,R_2)=(L_1R_2,L_2R_1)\top$; $(L_1,R_1)\top \cdot (L_2,R_2)\top=(L_1R_2,L_2R_1)$, $(L_1,R_1)\cdot (L_2,R_2)\top=(L_1L_2,R_1R_2)\top$, and $(L_1,R_1)(L_2,R_2)=(L_1R_1,L_2R_2)$. These rules define a group structure on $\TAut(A)$ and the group laws are compatible with the operation of $\top$ as the transpose operator. We have a short exact sequence of groups
\be \label{eq:exact}
1\to \Aut(A)\to \TAut(A)\to \ZZ/2 \to 1.
\ee
Now we have
\begin{proposition}\label{prop:split}
    A symmetric class $[A]$ contains a symmetric matrix, if and only if the exact sequence \ref{eq:exact} splits.
\end{proposition}

\begin{proof}
    The splitting of the sequence is equivalent to having an element $(L,R)\top\in TAut(A)$ of order $2$. If such an element exists, then $((L,R)\top)^2=(I,I)$ implies $LR=I$, or $R=L^\top$. Then from $A=LA^\top R^\top =LA^\top L$ we obtain that $L^\top A=A^\top L$ is symmetric. All implications can be reversed, so the converse also holds.
\end{proof}

We can say more. When the extension \eqref{eq:exact} splits, there are many ways of splitting the sequence, e.g. sections $s:\ZZ/2\to \TAut(A)$. There is a classical equivalence relation among sections, where two sections $s,s'$ are equivalent, if and only if there is an element $g\in \TAut(A)$ such that $s'(a)=gs(a)g^{-1}$. It is well-known that the equivalence classes are in bijection with the non-abelian cohomology set $H^1(\ZZ/2;\Aut(A))$, see \cite[p. 24 Exer 1]{NeukirchSchmidtWingberg2008} for discussion. As was shown in the proof above, each section yields a symmetric matrix $L^\top A$ where $(L,R)=s(1)$. We denote this matrix by $\sigma(s)$. In our case, the equivalence relation is classifying symmetric members of $[A]$ up to SH-equivalence. 

\begin{proposition}\label{prop:hom}
    Two sections $s,s':\ZZ/2\to \TAut(A)$ are equivalent, if and only if $\sigma(s),\sigma(s')$ are SH-equivalent.  
\end{proposition}

\begin{proof}
    Suppose that $s'(1)=gs(1)g^{-1}$. If $g\notin\Aut(A)$ then $gs(1)\in \Aut(A)$ and we may replace $g$ with $gs(1)$. Thus we are reduced to the case that $g\in \Aut(A)$. Let $g=(P,Q)$ and $s(1)=(L,R)\top$. Then $s'(1)=(PLQ^\top,QRP^\top)$ and $\sigma(s')=QL^\top P^\top A=QL^\top AQ^\top=Q\sigma(s)Q^\top$, where we have used that $P^\top A=AQ^\top$, since $(P,Q)$ is an automorphism. Hence $s'(1)\sim_{SH}s(1)$. The opposite direction is proved by the same equality, where the condition for it to hold is that $(P,Q)$ is an automorphism.
\end{proof}

\begin{corollary}\label{cor:hom}
    \begin{itemize}
        \item[(a)] The number of the SH-equivalence classes in $[A]$ is zero, or equals the size $|H^1(\ZZ/2;\Aut(A))|$.
        \item[(b)] If $\Aut(A)$ is abelian, the number of SH-equivalence classes in $[A]$ is zero or a power of $2$. 
    \end{itemize}
\end{corollary}
Part (a) is clear from the discussion above. If $\Aut(A)$ is abelian, the cohomology set is an abelian group, annihilated by $2$. Hence its size is a power of $2$, and (b) follows.

\subsection{Classifying symmetric and antisymmetric IW}
Given a symmetric class $[A]$, we wish now to classify all the (anti-)symmetric members in the class up to SH-equivalence. Without loss of generality, we shall assume that $A$ itself is (anti-)symmetric. A simple observation is that any (anti-)symmetric $S\in [A]$ is SH-equivalent to some $MA$ for a monomial $M$, for if $PAQ^\top$ is (anti-)symmetric, then its monomial conjugation $(Q^\top P)A$ is similarly so. Now, Proposition  \ref{prop:symm} tells us that $(M^\top,M)$ should be an automorphism of $A$. This reduces the question to the following: When is $MA\sim_{SH}A$ for an automorphism $(M^\top, M)\in \Aut(A)$? 

\begin{lemma}\label{lem:symm}
    Let $A$ be an (anti-)symmetric IW matrix. Then the (anti-)symmetric matrix $MA$ where $(M^\top,M)\in \Aut(A)$ is SH-equivalent to $A$, if and only if there exists an automorphism $(P,Q)\in \Aut(A)$, satisfying $Q^\top P=M$.
\end{lemma}

\begin{proof}
    Suppose first that $MA\sim_{SH} A$. Then $MA=TAT^\top$ for some $T$ and then it follows that $(P,Q):=(T^\top M,T^\top)\in \Aut(A)$, and further $Q^\top P=M$. Conversely if $Q^\top P=M$ we can define $T=Q$ and then $MA=TAT^\top$.
\end{proof}
More generally, if $(M_i^\top,M_i)\in \Aut(A)$, $i=1,2$, then $M_1A\sim_{SH} M_2A$ if and only if $\exists (P,Q)\in \Aut(A)$ such that $Q^\top M_1PM_1^\top =M_2M_1^\top$. This justifies the following classification algorithm (Algorithm \ref{alg:symm}) that inputs an IW matrix $A$, and outputs a list of representatives of (anti-)symmetric members of $[A]$ up to SH-equivalence. 

\begin{algorithm}
    \caption{Classifying (anti-)symmetric elements in $[A]$ up to SH-equivalence}\label{alg:symm}
    \begin{algorithmic}[1]
        \Procedure{\texttt{SymClasses}}{A}
            \State $A\gets A'$ for (anti-)symmetric $A'\in [A]$.
            \State $Rep\gets \{A\}$
            \For{$(T,U)\in\Aut(A)$}
                \If{$(T,U)=(M^\top,M)$}
                    \State $New\gets True$
                    \For{$(P,Q)\in \Aut(A)$}
                        \For{$B\in Rep$}
                            \State $M_1\gets BA^{-1}$
                            \If{$Q^\top M_1 P M_1^\top=M^\top M_1$}
                                \State $New\gets False$ 
                            \EndIf
                        \EndFor
                    \EndFor
                    \If{New}
                        \State $Rep\gets Rep\cup \{MA\}$
                    \EndIf
                \EndIf
            \EndFor\\
           \Return{Rep}
        \EndProcedure 
    \end{algorithmic}
\end{algorithm}

\subsection{Primitive and imprimitive classification}

Our SH classification algorithm relies on looping over the automorphism group, which could be large. A better strategy is to apply Algorithm \ref{alg:symm} only to primitive classes and deduce from this the classification in the general case. This is the content of Theorems \ref{thm:symm} and \ref{thm:a-symm}, which we can think of as the primitive decomposition theorems adapted to the symmetric and anti-symmetric cases. 
To ease the discussion, we will discuss in detail only the symmetric case. The anti-symmetric case is identical, except for a minus sign here and there. At the end of this section we will state the main result for the anti-symmetric case. 
The following proposition is an immediate corollary of Theorem \ref{thm:prim}.

\begin{proposition}
    Let $[A]$ be a symmetric IW class. Then 
    \be\label{eq:symmclass}  A\sim_H \bigoplus_{i=1}^r A_i^{\oplus m_i}\oplus \bigoplus_{i=1}^s B_i^{\oplus n_i}\oplus \bigoplus_{i=1}^t (C_i\oplus C_i^\top)^{\oplus o_i}\ee     
    where all blocks are primitive, $A_i$ are symmetric, $[B_i]$ are symmetric but do not contain a symmetric member, $[C_i]$ are non-symmetric, and all $A_i,B_i,C_i,C_i^\top$ are pairwise nonequivalent.$\Box$
\end{proposition}

Now let $[A]$ be a symmetric class. We may assume that $A$ is already in the form given in \eqref{eq:symmclass}. Let 
$$L\ := \ \bigoplus_{i=1}^r I^{\oplus m_i}\oplus \bigoplus_{i=1}^s L_i^{\oplus n_i}\oplus \bigoplus_{i=1}^t \begin{bmatrix}
    0 & I\\ I & 0
\end{bmatrix}^{\oplus o_i}, \text{and } R\ := \  \bigoplus_{i=1}^r I^{\oplus m_i}\oplus \bigoplus_{i=1}^s R_i^{\oplus n_i}\oplus \bigoplus_{i=1}^t \begin{bmatrix}
    0 & I\\ I & 0
\end{bmatrix}^{\oplus o_i},$$
such that $L_iB_iR_i^\top =B_i^\top$ for all $1\le i\le s$. 
In particular $LAR^\top=A^\top$.
To find a symmetric member of $[A]$, we need to find an automorphism $(M^\top L,MR)\in \Aut(A)$. Then $MA$ will be a symmetric member, and running over all such $M$ will produce an exhaustive list of all symmetric members in $[A]$ up to SH-equivalence. Theorem \ref{thm:wreath} tells us that $(M^\top L,MR)$ should preserve the each of the blocks $A_i^{\oplus m_i}$, $B_i^{\oplus n_i}$ $C_i^{\oplus o_i}$ and $C_i^{\top \oplus o_i}$. Thus $M,L,R$ preserve the index sets corresponding to the blocks $A_i^{\oplus m_i}$, $B_i^{\oplus n_i}$, and $(C_i+C_i^\top)^{\oplus o_i}$. It follows that any SH-class in $[A]$ contains a member which is a diagonal block sum partitioned according to these index sets. This justifies reduction of the search to each one of these blocks separately. We will now consider the following three cases.\\ 

\noindent {\large\bf Case I}: $A=A_0^{\oplus m}$ for an $m\times m$ primitive and symmetric IW matrix $A_0$. In this case $L=R=I$ and we are looking for an automorphism $(M^\top,M)$. $M$ is a block monomial matrix, with blocks that are $m\times m$ monomial matrices. Let us write uniquely $M=D|M|$ where $D$ is block diagonal with monomial blocks of size $m\times m$ and $|M|$ is a block permutation matrix with identity blocks $I_m$. Observe that $|M|=|M^\top|=|M|^\top$ since $A$ is block diagonal, showing that $|M|$ is an involution. Reorganizing the blocks in $A$ if necessary, $$|M|\ =\ I^{\oplus a}\oplus \begin{bmatrix}
    0 & I\\ I & 0
\end{bmatrix}^{\oplus b}, a+2b=m,$$ and 
\be\label{eq:M0} M\ = \  \bigoplus_{i=1}^a M_i \oplus  \bigoplus_{i=1}^b \begin{bmatrix}
    0 & P_i\\ Q_i & 0
\end{bmatrix}.\ee  
The condition that $(M^\top,M)$ is an automorphism is equivalent to $(M_i^\top,M_i)$ and $(P_i^\top,Q_i)$ are automorphisms of $A_0$ for all $i$. Finally we see that 
\be  
    MA=\bigoplus_{i=1}^a M_iA_0 \oplus \bigoplus_{i=1}^b \begin{bmatrix}
        0 & P_iA_0 \\ Q_iA_0 & 0
    \end{bmatrix}.
\ee
The condition $(P_i^\top,Q_i)\in \Aut(A_0)$ is necessary and sufficient for $MA$ to be symmetric.\\

\noindent {\large\bf Case II:} $A=A_0^{\oplus n}$, $[A_0]$ symmetric, not admitting a symmetric member, and $LA_0R^\top=A_0^\top$ for monomial $L,R$. Let $\Lambda=L^{\oplus n}$ and $\Sigma=R^{\oplus n}$. Then $\Lambda A\Sigma^\top=A^\top$. Suppose now that $(M^\top \Lambda,M\Sigma)\in \Aut(A)$. As before, $|M|$ is an involution, and after reorganizing the blocks, $M$ has the form as in \eqref{eq:M0} with $a=0$.
\be  
    MA=\bigoplus_{i=1}^{n/2} \begin{bmatrix}
        0 & P_iA_0 \\ Q_iA_0 & 0
    \end{bmatrix}.
\ee
with the condition that $(P^\top L,QR)$ is an automorphism of $A_0$. This condition is necessary and sufficient for $MA$ to be symmetric. Incidentally, $n$ has to be even.\\

\noindent {\large\bf Case III:} $A=(A_0\oplus A_0^\top)^{\oplus n}$ and $[A_0]$ is not symmetric. Let $L=R=\begin{bmatrix}
    0 & I \\ I & 0
\end{bmatrix}$, $\Lambda=L^{\oplus n}$ and $\Sigma=R^{\oplus n}$. Then $\Lambda A\Sigma^\top=A^\top$, and we are looking for a block monomial $M$ such that $ (M^\top\Lambda,M\Sigma)\in \Aut(A)$. In particular $(M^\top,M)$ is an isomorphism from $A^\top$ to $A$, and $(|M|^\top, |M|)$ is a block permutation keeping the diagonal. Thus $|M|=|M|^\top$ is an involution, and further switches between even and odd indices. Then after reorganizing the $A_0^\top$ blocks, we may assume that $|M|=\begin{bmatrix}
    0 & I\\ I & 0
\end{bmatrix}^{\oplus n}$ and $M$ has the same form as \eqref{eq:M0} with $a=0,b=n$. Hence 
\be  \label{eq:caseIII}
    MA=\bigoplus_{i=1}^n \begin{bmatrix}
        0 & P_iA_0 \\ Q_iA_0^\top & 0
    \end{bmatrix},
\ee
and the condition that $(M^\top\Lambda,M\Sigma)\in \Aut(A)$ is equivalent to $MA$ to be symmetric, which is equivalent to $(P_i^\top,Q_i)\in \Aut(A_0)$ for all $i$.\\

We can revise slightly cases I and II. We know in case II that the number of blocks should be even. Thus we can rewrite half of the blocks as $A_0^\top$ up to H-equivalence, and then the treatment is identical to case III. In case I we can simply write $A_0^\top=A_0$ since $A_0$ is symmetric. Therefore, the $2\times 2$ blocks in \eqref{eq:caseIII} are the unified form to write all of the $2\times 2$ blocks in cases I,II and III. The condition $(P_i^\top,Q_i)\in \Aut(A_0)$ is the same condition throughout. 
Moreover, each $2\times 2$ block in all cases is SH-equivalent to $\begin{bmatrix}
    0 & A_0\\ A_0^\top & 0
\end{bmatrix}$ by the symmetric monomial transformation $P_i\oplus I,P_i\oplus I$, and taking account of the relation $P_iA_0=A_0Q_i^\top$. The following proposition summarizes what we have proved so far.

\begin{proposition}\label{prop:class_sym}
    Let $A$ be an IW matrix, such that $[A]$ is a symmetric H-equivalence class. Write the primitive decomposition of $A$ as 
    $$A\sim_H \bigoplus_{i=1}^r A_i^{\oplus m_i}\oplus \bigoplus_{i=1}^s B_i^{\oplus 2n_i}\oplus \bigoplus_{i=1}^t (C_i\oplus C_i^\top)^{\oplus o_i}$$
    with $A_i$ symmetric, $[B_i]$ symmetric, not admitting a symmetric member, $[C_i]$ nonsymmetric, and all $A_i,B_i,C_i,C_i^\top$ are pairwise non H-equivalent.  
    Then every symmetric member of $[A]$ is SH-equivalent to a matrix of the form 
    \be\label{eq:symm_decomp}
         \Sigma\ :=\ \bigoplus_{i=1}^t \bigoplus_{j=1}^{a_i}M_{i,j}A_i \oplus \bigoplus_j \begin{bmatrix}
            0 & X_j\\ X_j^\top & 0
         \end{bmatrix},
    \ee
    where $a_i\le r_i$, $M_{i,j}$ are monomial, $M_{i,j}A_i$ are symmetric, and $X_j\in \{A_i,B_i,C_i\}$ exhausts half of the remaining $A_i$, half of the $B_i$, and all of the $C_i$. 
    $\Box$
    \end{proposition}
     For a square matrix $X$, let us denote by $S_X=\begin{bmatrix}
        0 & X\\X^\top & 0
    \end{bmatrix}$.  
    The blocks $M_{i,j}A_i$ in $\Sigma$ of eq. \eqref{eq:symm_decomp} will be called type I blocks. Likewise the blocks $S_X$ for a primitive symmetric $[X]$ are of type II, and those $S_X$ for primitive non symmetric $[X]$ are of type III.\\

    To every symmetric integer matrix $A$ we associate a weighted undirected graph $\mathcal{SG}(A)$ with $A$ as its adjacency matrix. If $A$ abd $B$ are SH-equivalent, the associated graphs $\mathcal{SG}(A)$ and $\mathcal{SG}(B)$ are isomorphic. If $X$ is block of type I,II or III, the associated graph $\mathcal{SG}(S_X)$ is connected. It follows that the decomposition above of $\Sigma$ corresponds to connected component decomposition of the graph $\mathcal{SG}(\Sigma)$.

    \begin{lemma}\label{lem:SX}
        For any two primitive IW matrices $X,X'$, $$X\sim_H X' \vee X^\top\sim_H X' \ \iff S_X\sim_{SH} S_{X'}.$$
    \end{lemma}

    \begin{proof}
        Suppose first that $X\sim_H X'$ by a pair $(L,R)$. Then applying $(L\oplus R,L\oplus R)$ to $S_X$ yields $S_{X'}$. When $X^\top\sim_H X'$, this similarly shows that $S_{X^\top}\sim_{SH} S_{X'}$. Since $S_X\sim_{SH} S_{X^\top}$ by $(S_I,S_I)$, then we transitively have $S_X\sim_{SH} S_{X'}$. This shows the `only if' part. Notice that for this part we did not use the primitivity or IW assumptions. Conversely, if $S_X\sim_{SH} S_{X'}$, then $X\oplus X^\top \sim_H S_X\sim_{H} S_{X'}\sim_H X'\oplus X^{'\top}$. By the primitive decomposition theorem we must have $X\sim_H X'$ or $X\sim_H X^{'\top}$. 
    \end{proof}
    Next, we need a proposition telling us what is the symmetric automorphism group of a $2\times 2$ block $S_X$, when $X$ is primitive.

    \begin{proposition} \label{prop:symautSx}
        If $X$ is a primitive IW, then $$\SAut(S_X)\ \cong \ \begin{cases}
            \TAut(X) & [X] \text{ is symmetric}\\
            \Aut(X) & [X] \text{ is not symmetric}
        \end{cases}. $$
    \end{proposition}

    \begin{proof} Let $G$ denote the group of $n\times n$ TH-operations and define a map $\varphi:G\to Mon(2n)^2$ by $\varphi((L,R))=(L\oplus R,L\oplus R)$ and $\varphi(\top)=(S_I,S_I)$. We can check that $\varphi$ is an injective group homomorphism, with image in the symmetric Hadamard operations. We can further check that $\varphi(\Aut(X))\subseteq \SAut(S_X)$.
        In the case where $[X]$ is non-symmetric, let us show that $\varphi|_{\Aut(X)}:\Aut(X)\to \SAut(S_X)$ is an isomorphism. Let $T_X:=S_XS_I=X\oplus X^\top$. By the primitive decomposition theorem, the automorphisms of $T_X$ are block diagonal, of the form $(P_1\oplus P_2,Q_1\oplus Q_2)$ where $(P_1,Q_1)\in \Aut(X)$ and $(P_2,Q_2)\in \Aut(X^\top)$. This implies that the automorphisms of $S_X$ are of the form $(P_1\oplus P_2,Q_2\oplus Q_1)$. The symmetric automorphisms of $S_X$ must further satisfy $P_2=Q_1$ and $Q_2=P_1$. Every such automorphism is in the image $\varphi$, and the injectivity of $\varphi$ finishes the proof in this case.\\ 

        Now suppose that $[X]$ is symmetric, the primitive decomposition theorem allows extra automorphisms of $T_X$ of block anti-diagonal form (same type on the left and the right) and any automorphism is either block diagonal or block anti-diagonal. The same conclusion passes to $S_X=S_IT_XS_I^\top$. One explicit block anti-diagonal symmetric automorphism is 
        $$ \xi\ := \ \left(\begin{bmatrix}
            0 & R\\ L & 0
        \end{bmatrix}, \begin{bmatrix}
            0 & R\\ L & 0
        \end{bmatrix} \right),$$ 
        for $LXR^\top=X^\top$. Any element of $\SAut(S_X)$ is either $U$ or $\xi \cdot U$ for a block diagonal symmetric automorphism $U$. The diagonal symmetric automorphisms are exactly those that are in $\varphi(\Aut(X))$. Thus $[\SAut(S_X):\varphi(\Aut(X))]=2$. But note that $\varphi(\top (L,R))=\xi$, which shows that $\varphi(\TAut(X))=\SAut(S_X)$, and the claim is proved  in this case either.
    \end{proof}

    We are now in the position to prove the following classification theorem of symmetric members in a class $[A]$. For any set $S$ and an equivalence relation $R$ on $S$, we say that elements $x_1,\ldots,x_m\in S$ form a \emph{classification list} of $R$, if every element in $S$ is $R$-equivalent to precisely one $x_i$. 

    \begin{theorem}[Symmetric case]\label{thm:symm}
        Suppose that $A$ is an IW-matrix, with the primitive decomposition 
        $$A\ \sim_H\ \bigoplus A_i^{\oplus m_i}\oplus \bigoplus B_i^{\oplus 2n_i} \oplus \bigoplus (C_i+C_i^\top)^{\oplus o_i},$$
        such that $A_i$ are symmetric, $[B_i]$ are symmetric, not having a symmetric member, and $[C_i]$ are not symmetric, and such that all $A_i,B_i,C_i,C_i^\top$ are pairwise non H-equivalent. For each $A_i$ let $\{A_i^j\}$ be a classification list of all symmetric members of $[A_i]$ up to SH-equivalence. For each $A_i$ choose nonnegative $a_i,b_i$ such that $a_i+2b_i=m_i$, and choose a multiset $\Lambda_i:=\{A_i^{j_1},\ldots,A_i^{j_{a_i}}\}$. Let $\Lambda=(\Lambda_i)$. Then the matrices
        \be\label{eq:class_sym} \Sigma_\Lambda\ := \bigoplus_{i,\Lambda_i} A_i^{j_k} \oplus \bigoplus \begin{bmatrix}
            0 & A_i\\
            A_i & 0
        \end{bmatrix}^{\oplus b_i}\oplus \begin{bmatrix}
            0 & B_i\\
            B_i^\top & 0
        \end{bmatrix}^{\oplus n_i} \oplus \begin{bmatrix}
            0 & C_i\\
            C_i^\top & 0
        \end{bmatrix}^{\oplus o_i}
        \ee
        form a classification list of all symmetric members of $[A]$ up to SH-equivalence. Let $\{X_i\},\{S_{Y_i}\},\{S_{Z_i}\}$ be the lists of distinct $1\times 1$ or $2\times 2$  blocks of $\Sigma_\Lambda$ of types I,II,III with multiplicities $p_i,q_i,r_i$ respectively. Then 
        \be 
            \SAut(\Sigma_\Lambda)\ \cong \ \prod_i \SAut(X_i)\wr S_{p_i} \times \prod_i \TAut(Y_i)\wr S_{q_i}\times \prod_i \Aut(X_i)\wr S_{r_i}.
        \ee
    \end{theorem}

    \begin{proof}
        By Proposition \ref{prop:class_sym}, every symmetric member of $[A]$ is SH-equivalent to one of the $\Sigma_\Lambda$. It remains to show that every the $\Sigma_\Lambda$ are pairwise non SH-equivalent. Since the blocks of $\Sigma_\Lambda$ correspond to the connected components of the graph $\mathcal{SG}(\Sigma_\Lambda)$, any SH-isomorphism between $\Sigma_\Lambda$ and $\Sigma_{\Lambda'}$ carries the $2\times 2$-blocks to $2\times 2$ blocks (since these blocks are not primitive), and consequently $1\times 1$ blocks to $1\times 1$ blocks. It follows that for each $i$, $\Lambda_i=\Lambda_i'$, because the $A_i^j$ are pairwise non SH-equivalent. This shows that $\Lambda=\Lambda'$, and since $\Sigma_\Lambda$ depends only on $\Lambda$, this completes the first part of the proof, that $\{\Sigma_\Lambda\}$ is a classification list.\\

        The blocks $X_i,Y_i,Z_i$ are pairwise non SH-equivalent, due to Lemma \ref{lem:SX}. Since they correspond to connected components of the graph $\mathcal{SG}(S_X)$, any symmetric automorphism of $\Sigma_\Lambda$ must take blocks to identical blocks. Then 
        $$\SAut(\Sigma_\Lambda) \ \cong \ \prod \SAut(X_i^{\oplus p_i})\times \prod \SAut(S_{Y_i}^{\oplus q_i})\times \prod \SAut(S_{Z_i}^{\oplus r_i}).$$
        Indeed, $\SAut(\Sigma_\Lambda)$ lies in this product, since it can only act among identical blocks. On the other hand every element of the product is a symmetric automorphism of the whole matrix. So our proof is reduced to the case where $\Sigma_\Lambda=U^{\oplus m}$ for a single symmetric block $U$ of type I,II or III. Now, $\SAut(U^{\oplus m})\cong \SAut(U)\wr S_m$, since any symmetric automorphism has to permute between the blocks, and can do so freely. Such an automorphism is a unique product of a block permutation, followed by a block diagonal automorphism. The claim for $U$ of type I is tautological. Proposition \ref{prop:symautSx} finishes the proof in cases II and III. 
    \end{proof}

    \subsection{The anti-symmetric case}
    We conclude this section by stating the anti-symmetric classification theorem. The proof is almost identical to the proof of Theorem \ref{thm:symm} and the propositions leading to it, except for a change of sign here and there. We define the \emph{anti-symmetric} double block
    $$A_X\ := \ \begin{bmatrix}
        0 & X\\ -X^\top & 0
    \end{bmatrix}.$$ 
    We continue to use the distinction to types I,II,III, defined analogously.

    \begin{theorem}[Anit-symmetric case]\label{thm:a-symm}
        Suppose that $A$ is an IW-matrix, with the primitive decomposition 
        $$A\ \sim_H\ \bigoplus A_i^{\oplus m_i}\oplus \bigoplus B_i^{\oplus 2n_i} \oplus \bigoplus (C_i+C_i^\top)^{\oplus o_i},$$
        such that $A_i$ are anti-symmetric, $[B_i]$ are symmetric, not having an anti-symmetric member, and $[C_i]$ are not symmetric, and such that all $A_i,B_i,C_i,C_i^\top$ are pairwise non H-equivalent. For each $A_i$ let $\{A_i^j\}$ be the classification list of all the anti-symmetric members of $[A_i]$ up to SH-equivalence. For each $A_i$ choose nonnegative $a_i,b_i$ such that $a_i+2b_i=m_i$, and choose a multiset $\Lambda_i:=\{A_i^{j_1},\ldots,A_i^{j_{a_i}}\}$. Let $\Lambda=(\Lambda_i)$. Then the matrices
        \be \Sigma_\Lambda\ := \bigoplus_{i,\Lambda_i} A_i^{j_k} \oplus \bigoplus \begin{bmatrix}
            0 & A_i\\
            -A_i & 0
        \end{bmatrix}^{\oplus b_i}\oplus \begin{bmatrix}
            0 & B_i\\
            -B_i^\top & 0
        \end{bmatrix}^{\oplus n_i} \oplus \begin{bmatrix}
            0 & C_i\\
            -C_i^\top & 0
        \end{bmatrix}^{\oplus o_i}
        \ee
        form a classification list of all the anti-symmetric members of $[A]$ up to SH-equivalence. Let $\{X_i\},\{A_{Y_i}\},\{A_{Z_i}\}$ be the lists of distinct $1\times 1$ or $2\times 2$ diagonal blocks of $\Sigma_\Lambda$ of types I,II,III with multiplicities $p_i,q_i,r_i$ respectively. Then 
        \be 
            \SAut(\Sigma_\Lambda)\ \cong \ \prod_i \SAut(X_i)\wr S_{p_i} \times \prod_i \TAut(Y_i)\wr S_{q_i}\times \prod_i \Aut(X_i)\wr S_{r_i}.
        \ee
    \end{theorem}

\subsection{Case study: Symmetric classification of projective space incidence and weighing matrices.} \label{sec:proj} In this subsection we show the consequences of the symmetric classification theorems: Proposition \ref{prop:symm} and Lemma \ref{lem:symm} to the finite projective space incidence matrix and weighing matrix. Let $\FF=\FF_q$ be the finite field of $q=p^r$ elements, $p$ an odd prime number. The projective space $\mathbb P^d(\FF)$ is the set of all lines through the origin in the space $V:=\FF^{d+1}$. A hyperplane $H\subset \mathbb P^d(\FF)$ is a set defined by a single nonzero homogenous linear equation on $V$. The \emph{projective space incidence matrix} $PI=PI_{d,q}$ is a $\{0,1\}$-matrix indexed by pairs of a points and an hyperplane, and $(PI)_{a,H}=1$ iff $a\in H$.\\ 

The projective space incidence matrix can be given concretely in terms of a nondegenerate bilinear form $[,]:V\times V\to \FF$ as follows: For each point $a$ and each hyperplane $H$, choose representing vectors $v=v(a),w=w(H)$ respectively, such that the hyperplane $H$ is given by the equation $[x,w]=0$. The incidence matrix is given by $(PI)_{a,H}=1$ iff $[v,w]=0$. The \emph{projective space weighing matrix}, denoted $PW=PW_{d,q}$ is given by $(PW)_{a,H}=0$ iff $[x,w]=0$, $(PW)_{a,H}=1$ iff $[v,w]\in (\FF^\times)^2$ is a quadratic residue, and $(PW)_{a,H}=-1$ iff $[v,w]$ is a quadratic nonresidue. It is easy to see that choosing another nondegenerate bilinear, different representatives, and changing the order of the points and hyperplanes, will altogether result in an H-equivalent matrices, for both $PW$ and $PI$. Therefore both matrices are uniquely defined up to H-equivalence. The weighing matrix $PW$ is a well known object (see e.g. \cite{JUNGNICKEL1999294,GD2023}).\\

Invertible linear maps $A:V\to V$ are well defined on the projective space, and induce permutations on the sets of points and hyperplanes. Let $\phi:V\to V$ be the $p$-power Frobenius map. A \emph{semilinear map} $S:V\to V$ is a composition $S=A\circ \phi^i$, where $A$ is linear. The set of semilinear maps is a group under composition, denoted by $G\Gamma L(V)$. The quotient $PG\Gamma L(V):=G\Gamma L(V)/scalars$ acts by permutations on the sets of points and hyperplanes in projective space, also preserves the incidence relation. We have $\Gamma L (V)=GL(V)\rtimes Gal(\FF/\FF_p)$ and $P\Gamma L (V)=PGL(V)\rtimes Gal(\FF/\FF_p)$. For a scalar or matrix $A$ and a Galois automorphism $\sigma$, we shall denote $A^\sigma$ the entrywise $\sigma$ action on $A$. We identify a linear map $A$ with its matrix with respect to the standard basis of $V$. Then for $A\in GL(V)$, $\phi^iA\phi^{-i}=A^{\phi^i}$. A semilinear map $S$ defines an automorphism of $PI$, via the pair of permutations $(\pi_L(S),\pi_R(S))$ on points and hyperplanes. The \emph{fundamental theorem of projective geometry} (see \cite[chap. 2]{Artin1988}) says that when $d\ge 2$, $\PermAut(PI)=PG\Gamma L(V)$.  A similar result for $PW$ we states:
\begin{proposition}\label{prop:PWAut}
    For $d\ge 2$,
    \be \label{eq:isomPW} \Aut(PW)\ = \ PG\Gamma L(V)^+:=G\Gamma L(V)/(\FF^\times)^2.\ee
\end{proposition}

\begin{proof}
    We have a map $G\Gamma L(V)\to \Aut(PW)$ mapping $S$ to the pair of monomials $(m_L(S),m_R(S))$ with underlying permutations $(\pi_L(S),\pi_R(S))$, and sign corrections that occur due to the fact that $S$ need not map representatives to representatives. The composition of maps $G\Gamma L(V)\to \Aut(PW)\stackrel{|\cdot|}{\to}\PermAut(PI)$ factors through the isomorphism $P\Gamma L(V)\cong \PermAut(PI)$. Hence $\Aut(PW)\to \PermAut(PI)$ is surjective, and its kernel is $\langle (-I,-I)\rangle$, since $PW$ is primitive. Notice that $(-I,-I)$ is the image of $\lambda I\in G\Gamma L(V)$ for a quadratic nonresidue scalar $\lambda$. This shows that $G\Gamma L(V)\to \Aut(PW)$ is surjective, and that its kernel is a subgroup of $\FF^\times$. Clearly this kernel is $(\FF^\times)^2$, which proves the proposition.
\end{proof}

Notice that $m_L,m_R$ are group homomorphisms $G\Gamma L(V)\to \Mon(d+1)$. From now on we shall work with the standard inner product $[-,-]$. Under this setting $PI$ and $PW$ are symmetric matrices. For each semilinear map $S=A\circ \phi^i$, $A$ linear, we define $S^\dagger :=\phi^{-i}\circ A^\top$, and it is clear that $[v,Sw]=[S^\dagger v,w]^{\phi^{i}}$. Let us define $\{v,w\}=[v,w]\mod (\FF^\times)^2$ and $\langle v,w\rangle=[v,w]\mod \FF^\times$. These are being used to define $PW$ and $PI$ respectively. For any semilinear map $S$, $\{v,Sw\}=\{S^\dagger v,w\}$ and $\langle v,Sw\rangle=\langle S^\dagger v,w\rangle$. The fundamental theorem of projective geometry takes now the following concrete form. Any $S\in G\Gamma L(V)$ defines an automorphism of $PI$ by the pair $(\pi_R((S^{-1})^\dagger,\pi_R(S)))$. A similar result holds for $PW$ with $m_R$. We now state and prove the symmetric classification theorem for finite field projective spaces.

\begin{theorem} Let $d\ge 2$, $q=p^r$, $p$ an odd prime. The number of SH-equivalence classes of symmetric matrices in $[PW]$ is $$S(PW)=2+M_{d,q}+N_{r},$$ where $M_{d,q}=1$ if $d$ is odd and $q\equiv 1\mod 4$, and $M_{d,q}=0$ otherwise. And where $N_r=2$ if $r$ is even and $N_r=0$ otherwise.\\

The number of SH-equivalence classes of symmetric matrices in $[PI]$ is $$S(PI)=2+M_d+\tfrac{1}{2}N_r,$$ where $M_d=0$ if $d$ is even and $M_d=1$ if $d$ is odd.
\end{theorem}

\begin{proof}
    We use Proposition \ref{prop:symm} and Lemma \ref{lem:symm} applied to $PW$. Any symmetric member in $PW$ is SH-equivalent to $M\cdot PW$ for some monomial $M$ such that $(M^{-1},M)$ is an automorphism. There exists an $S\in P\Gamma L(V)$ such that $M=m_R(S)$, and $M^{-1}=m_L(S)=m_R((S^{-1})^\dagger)\implies M=m_R(S^\dagger)$. By the isomorphism $\eqref{eq:isomPW}$ it follows that $S^\dagger = S \mod (\FF^\times)^2$. The matrix $M\cdot PW$ is given by the symmetric form $\{u,Sw\}$. In addition $M_1\cdot PW$ and $M_2\cdot PW$ are SH-equivalent, if and only if for some monomial $T$ we have that $(M_2^{-1} TM_1,T)\in \Aut(PW)$. It follows that $T=m_R(L)$ for a semilinear $L$, writing $M_i=m_R(S_i)$ and substituting in $M_2^{-1}TM_1=m_R((T^{-1})^\dagger)$, we obtain the equality $LS_1L^\dagger=S_2\mod (\FF^\times)^2$. Rolling back the equations, any semilinear map $S$ satisfying $S^\dagger =S \mod (\FF^\times)^2$ defines symmetric matrix $M\cdot PW$ given by the symmetric form $\{v,Sw\}$ for $M=m_R(S)$ and $LS_1L^\dagger=S_2\mod (\FF^\times)^2$ ensures that we get SH-equivalent matrices. Thus we have shown that the symmetric classification of $[PW]$ is in bijection with the set $$SemSym\ := \{S\in P\Gamma L(V)\ | \ S^\dagger =\lambda S, \ \lambda\in (\FF^\times)^2\}/\equiv,$$ where $S_1\equiv S_2$ iff there is an $L\in P\Gamma L(V)$ such that $LS_1L^\dagger=S_2\mod (\FF^\times)^2$.\\
    
    Let us now analyze the set $SemSym$. Writing $S=A\phi^i$ for $A$ linear, the condition $S^\dagger=S\mod (\FF^\times)^2$ becomes $\phi^{-i}A^\top=\lambda A\phi^i$ for some $\lambda\in (\FF^\times)^2$ which is equivalent to  $\phi^{-i} A^\top\phi^{-i} =\lambda A$, where necessarily $\phi^{-2i}=id$, because the left hand side should be linear. Let $\tau$ be the unique Galois automorphism of order $2$ (if exists). Then $\phi^i=id$ or $\tau$. Write $\bar A=\phi^{-i} A\phi^{i}$ for a matrix or a scalar. The conditions on $A$ become $\bar A^\top=\lambda A$ and $\lambda\bar\lambda=1$. We now separate the discussion to three cases. \\

     {\bf\large Case I (The euclidean case):} $r$ is odd, and $d$ is even or $q\equiv 3\mod 4$: Since $r$ is odd, $\phi^i=id$, $\bar A=A$, and $\lambda^2=1$, $\lambda$ a square. If $d$ is even we cannot have $\lambda=-1$, or else $\det(A)=0$, since the dimension of $A$ is odd. If $q\equiv 3\mod 4$ then $-1$ is not a square, and again $\lambda\neq -1$. Thus we must have $\lambda=1$ and $A$ is symmetric. By the well-known classification of symmetric bilinear forms over finite fields, up to congruence there are only two cases, $A= B^\top I B$ or $A=B^\top \diag(\rho,1,\ldots,1)B$ for a nonsquare scalar $\rho\in \FF^\times$. Hence we may assume that $S=S_1:=A=I$ or $S=S_2:=A=\diag(\rho,1,\ldots,1)$. We will show that $S_2\not\equiv I=S_1$. Otherwise $S_2=L^\dagger L$, and writing $L=B\phi^i$ for $B$ linear, $S_2=\phi^{-j}B^\top B\phi^j=(\phi^{-j}B\phi^{j})^\top((\phi^{-j}B\phi^{j}))$. Taking the determinant we see that $\det S_2$ is a square, which is a contradiction. We conclude that in this case $SemSym$ is a set of two elements.\\

    {\bf\large Case II (The symplectic case):} $r$ is odd, $d$ is odd, and $q\equiv 1\mod 4$: In this case, in addition to the case that $A$ is symmetric, we may also have the case $A^\top=-A$. In this event, the classification of symplectic forms over fields implies that for some $B\in GL(V)$,$$A=B^\top J B, \ \ \ J=\begin{bmatrix}
        0 & -1\\ 1 & 0
    \end{bmatrix}^{\oplus (d+1)/2}.$$ Hence in addition we have $S=S_3:=J$, and $S_3$ is not equivalent in $SemSym$ to $S_1,S_2$, since $L^\dagger S_iL$, $i=1,2$ are symmetric. In this case $SemSym$ is a set of $3$ elements.\\

    {\bf\large Case III (The hermitian case):} $r$ is even: If $\phi^i=id$, then we are in a situation which is similar to cases I or II, and as before, that event contributes to SemSym either $2$ or $3$ elements, according to whether $d$ is odd and $q\equiv 1\mod 4$, or not. But now $\phi^i=\tau$ contributes new elements to $SemSym$, where $S=A\tau$. Such $S$ is not equivalent to $S_1,S_2,S_3$, since is not linear over $\FF$. We have that $\bar A^\top=\lambda A$ and $\lambda\bar\lambda=1$. We claim that there exists a $\nu\in \FF$ such that $\nu/\bar \nu=\lambda$. This is a consequence of the Hilbert 90 theorem, but let us show this concretely. If $\lambda=-1$, we let $\nu\in \FF$ be a square root of a non-square element of $\FF_{p^{r/2}}$. Otherwise we set $\nu=1+\lambda$, and compute $\nu/\bar\nu=(1+\lambda)/(1+\bar\lambda)=\lambda(1+\lambda)/\lambda(1+\bar\lambda)=\lambda$. Letting $A_0=\nu A$, $\bar A_0^\top=A_0$ is hermitian and $S=\nu^{-1}A_0\tau$. By the classification theorem for hermitian forms over finite fields, we have $A_0=\bar U^\top U$ for some $U\in GL(V)$. Then it follows that $S=(U\tau)^\dagger \bar\nu^{-1}\tau (U\tau)$, and in particular $S\equiv \alpha \tau$ for a scalar $\alpha$. Each form $\{u,\alpha\tau v\}$ is symmetric (note that $\bar\alpha/\alpha$ is always a square.) To finish the analysis we need to study when is $\alpha \tau\equiv \beta \tau$. Suppose that $\alpha$ and $\beta$ satisfy that for some $L$ and a square $\mu^2$, $L^\dagger \alpha \tau L=\beta \mu^2\tau$. Writing $L=B\phi^j$ for linear $B$ this is equivalent to $\alpha B^\top \bar B =\beta^{\phi^j}\nu^2I$ for some linear $B$ and a scalar $\nu$, which is equivalent to $\alpha/(\nu^2\beta^{\phi^j})I$ is hermitian. This in turn is equivalent to $\alpha/\beta^{\phi^j}\in (\FF^\times)^2\FF_{p^{r/2}}^\times$. Since $\FF_{p^{r/2}}^\times\subset (\FF^\times)^2$ and $\beta^\psi=\beta \mod (\FF^\times)^2$, then $\alpha \tau\equiv \beta\tau$ iff $\alpha=\beta \mod (\FF^\times)^2$. We conclude that in this case there are 2 additional SH-equivalence classes. This completes the proof for $[PW]$.\\

    The proof of $[PI]$ is similar, and we focus on the differences. We replace the form $\{u,v\}$ with the form $\langle u,v \rangle$. 
    The classification of the SH-equivalence classes is in bijection with the set 
    $$ISemSym\ := \ \{S\in P\Gamma L(V)\ | \ S^\dagger =\lambda S, \ \lambda\in \FF^\times\}/\approx,$$ where $S_1\approx S_2$ if and only if there exist a semilinear map $L$ such that $L^\dagger S_1 L=S_2 \mod \FF^\times$. If $S^\dagger=\lambda S$, then again $S=A$ or $S=A\tau$ for linear $A$ satisfying $\bar A^\top=\lambda A$ for $\lambda\bar\lambda=1$. There are three cases to consider. In Case A, $r$ is odd and $d$ is odd. Then $S=A$, and $A^\top=\pm A$. If $A^\top=A$, there are two $\approx$ classes determined by whether $\det A$ is a square or not. Notice that $\det(\lambda I)=\lambda^{d+1}$ is a square. There is another class coming from the case $A^\top=-A$. Case B is when $r$ is odd and $d$ is even. Then for $A^\top=A$ there is only one $\approx$ class, since $\det(\lambda I)$ could be a nonsquare. Finally case C when $r$ is even, adds one additional class coming from $S=A\tau$ for $\bar A^\top=\lambda A$. Like before we may reduce to $A=\alpha\tau$, but now there is only one $\approx$ class, since $\alpha\tau\approx \beta\tau$ for all $\alpha$ and $\beta$, as we do not care about squares.
\end{proof}

\begin{remark}
    The matrix $-PW$ results from the form $\{v,\lambda w\}$ for quadratic nonresidue $\lambda$.
    When $d$ is even it is not symmetrically equivalent to $PW$. However, when $d$ is odd $-PW\sim_{SH} PW$.
\end{remark}

\section{Counting IW matrices}\label{sec:6}
In this section we explain how to count the number of all $IW(n,k), SIW(n,k),$ $AIW(n,k)$, and similarly for weighing matrices, using the classification and the automorphism groups. We can do the counting class by class. Simply, if we want to count the size of a class $[A]$, we can use the simple formula

\be \label{eq:IWcountClass}  \big|[A]\big|\ = \ \frac{2^{2n}n!^2}{|\Aut(A)|}.\ee
This is the orbit-stabilizer formula, acknowledging that $[A]$ is the orbit of $A$ under $Mon(n)\times Mon(n)$. The total number of $IW(n,k)$ can de derived from the classification, yielding 

\be \label{eq:IWcount}
    \big| IW(n,k)\big|\ = 2^{2n}n!^2\sum_{A\in Cl(n,k)} \frac{1}{|Aut(A)|},
\ee
where $Cl(n,k)$ is a classification list of $IW(n,k)$ modulo H-equivalence. To compute this, we need to generate $Cl(n,k)$, which may be large, due to the large number of non-primitive classes. We can reduce the amount of work substantially if we base our analysis on the \emph{primitive} classification lists.\\

We define the \emph{characteristic} counting function of weight $k$ as the formal power series

$$Z_k(t) \ := \ 1+\sum_{n=1}^\infty \frac{\big| IW(n,k) \big|}{2^{2n}n!^2}t^n.$$
Let $PCl(n,k)$ be a classification list of the primitive $IW(n,k)$. Let 
$$ PZ_k(t) \ := \ \sum_{n=1}^\infty \sum_{A\in PCl(n,k)} \frac{t^n}{|\Aut(A)|},$$
which is the primitive analog of $Z_k(t)$. The next proposition bridges the gap between primitive and general counting.

\begin{proposition}\label{prop:countIW}
    As formal power series,
    \be\label{eq:ZetaIW} Z_k(t) \ = \ \exp\big(PZ_k(t)\big).\ee
\end{proposition}

\begin{proof} Let
     $\{A_1,A_2,\ldots,A_q\}=\bigcup_{m\le n} PCl(m,k)$, a classification list of all primitive IW of size less than or equal to $n$ up to H-equivalence, where $A_i$ has size $m_i$. By the primitive decomposition theorem, the classification list of $W(n,k)$ comprises of all block sums $A:=A_1^{\oplus i_1}\oplus A_2^{\oplus i_2}\oplus\cdots \oplus A_q^{\oplus i_q}$ with $i_1m_1+\cdots+i_qm_q=n$. The size the Automorphism group is $|\Aut(A)|=\prod_j |\Aut(A_j)|^{i_j}i_j!$, by Theorem \ref{thm:wreath}. It follows that 
     $$ \frac{\big| [A] \big|}{2^{2n}n!^2}=\frac{1}{|\Aut(A)|}=\frac{1}{\prod_j |\Aut(A_j)|^{i_j}i_j!},$$ and summing over all possible choices of $A_{i_j},m_j$ we obtain
     $$\frac{\big| IW(n,k)\big|}{2^{2n}n!^2}=\sum_{m_1i_1+\cdots m_qi_q=n} \frac{1}{\prod_j |\Aut(A_j)|^{i_j}i_j!}.$$ 
    This is the $n$th coefficient of the product series $\prod_i\exp(t^{m_i}/|\Aut(A_i)|)$, so all in all, $Z_k(t)$ and 
     $$Z_k(t)\ = \ \prod_{M\in PCL(k)} \exp\left(\frac{t^{m_M}}{|\Aut(M)|}\right)=\exp\big(PZ_k(t)\big).$$ Here $PCL(k)=\bigcup_m PCL(m,k)$ and $m_M$ is the size of $M$.
\end{proof}

\begin{remark}
    Equation \eqref{eq:ZetaIW} may be viewed as an Euler product formula for $Z_k(t)$, where the primitive IW classes are the primes.
\end{remark}

\begin{remark}
    If we wish to count only weighing matrices, we use the same formula, taking into account only primitive $W(m,k)$ classification. The same is true for any other limited set of allowed the entries of the matrix, as long as this set is stable under negation.
\end{remark}

\begin{example}
    Let us count the number of $IW(4,25)$. We compute pairs $(m,p)$ for each primitive class of $IW(m,4)$, $m\le 4$, where $m$ is the size and $p$ is the size of the automorphism group. The primitive classification tells us that $(m,p)=$ $(1,2)$, $(2,4)$, $(4,8)$ and $(4,16)$. Consequently $PZ_{25}(t)=\tfrac{t}{2}+\tfrac{t^2}{4}+\tfrac{t^4}{8}+\tfrac{t^4}{16}+\cdots.$ Hence $Z_{25}(t)=1 + \tfrac{t}{2} + \tfrac{3t^2}{8} + \tfrac{7t^3}{48} + \tfrac{97t^4}{384}+\cdots$. We conclude that $|IW(4,25)|=2^84!^2\cdot \tfrac{97}{384}=37,248$.
\end{example}

\subsection{Counting symmetric and anti-symmetric IW}
The idea here is similar. We use Theorems \ref{thm:symm} and \ref{thm:a-symm} as the analog of the primitive decomposition theorem. The only thing is that we need to understand what are the `primitive' blocks here, which are those of types I,II,and III. The starting point is the ordinary classification of primitive IW, plus the (anti-)symmetric classification primitive $IW$. We define the two counting series:

$$Z^S_k(t) \ := \ 1+\sum_{n=1}^\infty  \frac{|SIW(n,k)|}{2^nn!}, \text{ and } Z^A_k(t) \ := \ 1+\sum_{n=1}^\infty  \frac{|AIW(n,k)|}{2^nn!}.$$
Let $SPCl(n,k)$ (resp. $APCl(n,k)$) to be a classification list of all primitive symmetric (resp. anti-symmetric) $IW$ matrices of size $n$ and weight $k$ up to symmetric equivalence. Define the primitive counting series

$$PZ^S_k(t) \ := \ \sum_{n=1}^\infty \sum_{A\in SPCl(n,k)}\frac{t^n}{|\SAut(A)|}+\sum_{A\in PCl(n,k)} \frac{t^{2n}}{|\TAut(A)|}$$
and 
$$PZ^A_k(t) \ := \ \sum_{n=1}^\infty \sum_{A\in APCl(n,k)}\frac{t^n}{|\SAut(A)|}+\sum_{A\in PCl(n,k)} \frac{t^{2n}}{|\TAut(A)|}$$
Then we have:

\begin{proposition}\label{prop:count_sym_IW}
    As formal power series,
    \be Z^S_k(t)\ = \exp\left( PZ^S_k(t)\right) \text{ and } \ Z^A_k(t)\ = \exp\left( PZ^A_k(t)\right).
    \ee
\end{proposition}

\begin{proof}
    Analogous to the proof of Proposition \ref{prop:countIW}. We use the (anti-)symmetric classification theorems \ref{thm:symm} and \ref{thm:a-symm}. Note that for type III blocks $S_X$, $\TAut(X)=Aut(X)$.
\end{proof}

\begin{example}
    Let us count the number of symmetric and anti-symmetric $IW(4,25)$. For each primitive symmetric class, let us write a pair $\langle m,q\rangle$ for each $A\in SPCl(m,25)$ where $q$ is the size of $SAut(A)$. We write a pair $(m,p)$ for a usual primitive $IW(m,k)$ class, but now taking $p=|\TAut(A)|$. The symmetric primitive classification yields the pairs $\langle 1,2\rangle$, $\langle 2,2\rangle$, $\langle 4,4\rangle$, $\langle 4,8\rangle$ (four times),$\langle 4,2\rangle$, and $(1,4),(2,8)$. This leads to $PZ^S_{25}(t)=(\tfrac{t}{2}+\tfrac{t^2}{2}+\tfrac{t^4}{4}+\tfrac{4t^4}{8}+\tfrac{t^4}{2})+(\tfrac{t^2}{2}+\tfrac{t^4}{8})+\cdots$, and $Z^S_{25}(t)=1+\tfrac{t}{2}+\tfrac{9t^2}{8}+\tfrac{25t^3}{48}+\tfrac{769t^4}{384}+\cdots$. Hence $|SIW(4,25)|=2^44!\cdot \frac{769}{384}=769$. The classification yields no primitive anti-symmetric matrices. Therefore the pairs are only of type II and III (in fact of type II only), and they are $(1,4)$ and $(2,8)$. Thus $PZ^A_{25}(t)=\tfrac{t^2}{4}+\tfrac{t^4}{8}+\cdots$, and $Z^A_{25}(t)=1+\tfrac{t^2}{4}+\tfrac{5t^4}{32}$, and $|AIW(4,25)|=2^44!\cdot\tfrac{5}{32}=60$.
\end{example}

\section{Results for $IW(7,25)$}\label{sec:7}
In this section we outline the classification results of $IW(7,25)$. We supply: 

\begin{itemize}
    \item A classification list of all primitive $IW(m,25)$ for $m\le 7$.
    \item Say which pairs of them are transpose equivalent.
    \item The automorphism group id of each primitive class.
    \item The cardinality of each primitive class.
    \item The number of symmetric classes in each primitive class.
    \item The total count of the symmetric members in each primitive class.
    \item The total count of $IW(7,25)$.
    \item The total count of $SIW(7,25)$.
\end{itemize}

We first give an exhaustive classification list of all $IW(m,25)$ for $m\le 7$, up to TH-equivalence. Whenever a class contains a symmetric member, we chose to use it as a representative. The classes are enumerated by $m.q$, where $m$ is the size and $q$ is the serial number. Below the list we specify which are the non-symmetric primitive classes. 

\subsection{Classification list of primitive $IW(m,25)$, $m\le 7$}

\begin{itemize}
    \item[{\large \bf --$IW(1,25)$:}] $$A_1=[5].$$
    \item[{\large \bf --$IW(2,25)$:}] $$B_1=\begin{bmatrix}
        3 & 4 \\ 4 & -3
    \end{bmatrix}.$$
    \item[{\large\bf --$IW(3,25)$:}] $$\emptyset.$$
    \item[{\large \bf --$IW(4,25)$: $C_i=$}] 
    $$\left[\begin{array}{rrrr}
-4 & -2 & -2 & -1 \\
-2 & -1 & 4 & 2 \\
-2 & 4 & -1 & 2 \\
-1 & 2 & 2 & -4
\end{array}\right], \left[\begin{array}{rrrr}
-4 & -2 & -2 & -1 \\
-2 & -1 & 4 & 2 \\
-2 & 4 & 1 & -2 \\
-1 & 2 & -2 & 4
\end{array}\right]$$

    \item[{\large \bf --$IW(5,25)$: $D_i=$}]
            $$\left[\begin{array}{rrrrr}
-4 & -2 & -2 & -1 & 0 \\
-2 & 0 & 2 & 4 & -1 \\
-2 & 2 & 3 & -2 & 2 \\
-1 & 4 & -2 & 0 & -2 \\
0 & -1 & 2 & -2 & -4
\end{array}\right], \left[\begin{array}{rrrrr}
-3 & -2 & -2 & -2 & -2 \\
-2 & -3 & 2 & 2 & 2 \\
-2 & 2 & -3 & 2 & 2 \\
-2 & 2 & 2 & -3 & 2 \\
-2 & 2 & 2 & 2 & -3
\end{array}\right]$$

    \item[{\large \bf --$IW(6,25)$:}] 
    $E_i=$\\[0.3cm]
    \scalebox{0.5}[0.5]{$%
    \left[\begin{array}{rrrrrr}
-4 & -2 & -2 & -1 & 0 & 0 \\
-2 & -1 & 4 & 2 & 0 & 0 \\
-2 & 4 & 0 & 0 & -2 & -1 \\
-1 & 2 & 0 & 0 & 4 & 2 \\
0 & 0 & -2 & 4 & -1 & 2 \\
0 & 0 & -1 & 2 & 2 & -4
\end{array}\right], \left[\begin{array}{rrrrrr}
-4 & -2 & -2 & -1 & 0 & 0 \\
-2 & -1 & 4 & 2 & 0 & 0 \\
-2 & 4 & 0 & 0 & -2 & -1 \\
-1 & 2 & 0 & 0 & 4 & 2 \\
0 & 0 & -2 & 4 & 1 & -2 \\
0 & 0 & -1 & 2 & -2 & 4
\end{array}\right], \left[\begin{array}{rrrrrr}
2 & 0 & -2 & -4 & 1 & 0 \\
0 & -2 & 3 & -2 & -2 & 2 \\
-2 & 3 & 2 & -2 & 0 & -2 \\
-4 & -2 & -2 & -1 & 0 & 0 \\
1 & -2 & 0 & 0 & -2 & -4 \\
0 & 2 & -2 & 0 & -4 & 1
\end{array}\right], \left[\begin{array}{rrrrrr}
-4 & -2 & -2 & -1 & 0 & 0 \\
-2 & 0 & 3 & 2 & -2 & -2 \\
-2 & 3 & 0 & 2 & 2 & 2 \\
-1 & 2 & 2 & -4 & 0 & 0 \\
0 & -2 & 2 & 0 & -1 & 4 \\
0 & -2 & 2 & 0 & 4 & -1
\end{array}\right],$}\\[0.3cm]

\scalebox{0.5}[0.5]{$%
\left[\begin{array}{rrrrrr}
-4 & -2 & -2 & -1 & 0 & 0 \\
-2 & 0 & 4 & 0 & -2 & -1 \\
-2 & 4 & 0 & 0 & 1 & 2 \\
-1 & 0 & 0 & 4 & 2 & -2 \\
0 & -2 & 1 & 2 & 0 & 4 \\
0 & -1 & 2 & -2 & 4 & 0
\end{array}\right], \left[\begin{array}{rrrrrr}
-4 & -2 & -2 & -1 & 0 & 0 \\
-2 & 2 & 2 & 0 & -3 & -2 \\
-2 & 2 & 2 & 0 & 2 & 3 \\
-1 & 0 & 0 & 4 & 2 & -2 \\
0 & -3 & 2 & 2 & -2 & 2 \\
0 & -2 & 3 & -2 & 2 & -2
\end{array}\right], \left[\begin{array}{rrrrrr}
-3 & -3 & -2 & -1 & -1 & -1 \\
-3 & 1 & 1 & -1 & 2 & 3 \\
-2 & 1 & 3 & 1 & -3 & -1 \\
-1 & -1 & 1 & 3 & 3 & -2 \\
-1 & 2 & -3 & 3 & -1 & 1 \\
-1 & 3 & -1 & -2 & 1 & -3
\end{array}\right], \left[\begin{array}{rrrrrr}
-3 & -3 & -2 & -1 & -1 & -1 \\
-3 & 1 & 3 & -2 & 1 & 1 \\
-2 & 1 & -1 & 3 & -1 & 3 \\
-1 & -1 & 1 & 3 & 3 & -2 \\
-1 & 2 & 1 & 1 & -3 & -3 \\
-1 & 3 & -3 & -1 & 2 & -1
\end{array}\right],$}\\[0.3cm]

\scalebox{0.5}[0.5]{$%
\left[\begin{array}{rrrrrr}
-3 & -3 & -2 & -1 & -1 & -1 \\
-3 & 1 & 3 & -2 & 1 & 1 \\
-2 & 3 & -3 & 1 & 1 & 1 \\
-1 & -1 & 1 & 3 & -2 & 3 \\
-1 & -1 & 1 & 3 & 3 & -2 \\
-1 & 2 & 1 & 1 & -3 & -3
\end{array}\right],
 \left[\begin{array}{rrrrrr}
-3 & -3 & -2 & -1 & -1 & -1 \\
-3 & 2 & 3 & -1 & -1 & -1 \\
-2 & 3 & -3 & 1 & 1 & 1 \\
-1 & -1 & 1 & -2 & 3 & 3 \\
-1 & -1 & 1 & 3 & -2 & 3 \\
-1 & -1 & 1 & 3 & 3 & -2
\end{array}\right], \left[\begin{array}{rrrrrr}
-3 & -3 & -2 & -1 & -1 & -1 \\
-3 & 3 & -1 & -1 & 1 & 2 \\
-2 & -1 & 3 & 1 & 3 & -1 \\
-1 & -1 & 1 & 3 & -2 & 3 \\
-1 & 1 & 3 & -2 & -3 & -1 \\
-1 & 2 & -1 & 3 & -1 & -3
\end{array}\right], \left[\begin{array}{rrrrrr}
-3 & -2 & -2 & -2 & -2 & 0 \\
-2 & -2 & 0 & 2 & 3 & -2 \\
-2 & 0 & 2 & 3 & -2 & 2 \\
-2 & 2 & 3 & -2 & 0 & -2 \\
-2 & 3 & -2 & 0 & 2 & 2 \\
0 & -2 & 2 & -2 & 2 & 3
\end{array}\right],$}\\[0.3cm]

\scalebox{0.5}[0.5]{%
$\left[\begin{array}{rrrrrr}
-3 & -3 & -2 & -1 & -1 & -1 \\
-3 & 3 & -1 & -1 & 1 & 2 \\
-2 & 1 & 1 & 3 & 1 & -3 \\
-1 & -2 & 3 & -1 & 3 & 1 \\
-1 & -1 & 1 & 3 & -2 & 3 \\
-1 & 1 & 3 & -2 & -3 & -1
\end{array}\right]$}\\

\item[{\large \bf --$IW(7,25)$:}] $F_i=$\\[0.1cm]
\scalebox{0.5}[0.5]{$%
\left[\begin{array}{rrrrrrr}
-4 & -2 & -2 & -1 & 0 & 0 & 0 \\
-2 & 3 & 0 & 2 & 2 & 0 & -2 \\
-2 & 0 & 4 & 0 & -2 & -1 & 0 \\
-1 & 2 & 0 & 0 & 0 & 2 & 4 \\
0 & 2 & -2 & 0 & -4 & 0 & -1 \\
0 & 0 & -1 & 2 & 0 & -4 & 2 \\
0 & -2 & 0 & 4 & -1 & 2 & 0
\end{array}\right], \left[\begin{array}{rrrrrrr}
2 & -1 & -1 & -4 & 1 & 1 & 1 \\
-1 & 0 & 3 & -2 & 1 & -3 & -1 \\
-1 & 3 & 0 & -2 & -3 & 1 & -1 \\
-4 & -2 & -2 & -1 & 0 & 0 & 0 \\
1 & 1 & -3 & 0 & 1 & -2 & -3 \\
1 & -3 & 1 & 0 & -2 & 1 & -3 \\
1 & -1 & -1 & 0 & -3 & -3 & 2
\end{array}\right], \left[\begin{array}{rrrrrrr}
2 & -1 & -1 & -4 & 1 & 1 & 1 \\
-1 & -1 & 4 & -2 & -1 & -1 & -1 \\
-1 & 4 & -1 & -2 & -1 & -1 & -1 \\
-4 & -2 & -2 & -1 & 0 & 0 & 0 \\
1 & -1 & -1 & 0 & 2 & -3 & -3 \\
1 & -1 & -1 & 0 & -3 & 2 & -3 \\
1 & -1 & -1 & 0 & -3 & -3 & 2
\end{array}\right], \left[\begin{array}{rrrrrrr}
-4 & -2 & -2 & -1 & 0 & 0 & 0 \\
-2 & 1 & 2 & 2 & -2 & -2 & -2 \\
-2 & 2 & 1 & 2 & 2 & 2 & 2 \\
-1 & 2 & 2 & -4 & 0 & 0 & 0 \\
0 & -2 & 2 & 0 & -3 & 2 & 2 \\
0 & -2 & 2 & 0 & 2 & -3 & 2 \\
0 & -2 & 2 & 0 & 2 & 2 & -3
\end{array}\right],$}\\[0.3cm]

\scalebox{0.5}[0.5]{$%
\left[\begin{array}{rrrrrrr}
-4 & -2 & -2 & -1 & 0 & 0 & 0 \\
-2 & 1 & 2 & 2 & -2 & -2 & -2 \\
-2 & 2 & 2 & 0 & 0 & 2 & 3 \\
-1 & 2 & 0 & 0 & 4 & 0 & -2 \\
0 & -2 & 0 & 4 & 1 & 2 & 0 \\
0 & -2 & 2 & 0 & 2 & -3 & 2 \\
0 & -2 & 3 & -2 & 0 & 2 & -2
\end{array}\right], \left[\begin{array}{rrrrrrr}
-4 & -2 & -2 & -1 & 0 & 0 & 0 \\
-2 & 1 & 3 & 0 & -3 & -1 & -1 \\
-1 & -1 & 1 & 4 & 1 & -1 & 2 \\
-1 & 0 & 3 & -2 & 3 & 1 & 1 \\
-1 & 1 & 0 & 2 & 1 & 3 & -3 \\
-1 & 3 & -1 & 0 & -1 & 2 & 3 \\
-1 & 3 & -1 & 0 & 2 & -3 & -1
\end{array}\right], \left[\begin{array}{rrrrrrr}
1 & 1 & -1 & -4 & -2 & 1 & 1 \\
1 & -1 & 0 & -2 & 1 & -3 & -3 \\
-1 & 0 & 3 & -2 & 3 & 1 & 1 \\
-4 & -2 & -2 & -1 & 0 & 0 & 0 \\
-2 & 1 & 3 & 0 & -3 & -1 & -1 \\
1 & -3 & 1 & 0 & -1 & 3 & -2 \\
1 & -3 & 1 & 0 & -1 & -2 & 3
\end{array}\right], \left[\begin{array}{rrrrrrr}
1 & -1 & 1 & -4 & -1 & 2 & 1 \\
-1 & 4 & -1 & -2 & 1 & -1 & 1 \\
1 & -1 & 0 & -2 & 1 & -3 & -3 \\
-4 & -2 & -2 & -1 & 0 & 0 & 0 \\
-1 & 1 & 1 & 0 & 2 & 3 & -3 \\
2 & -1 & -3 & 0 & 3 & 1 & 1 \\
1 & 1 & -3 & 0 & -3 & 1 & -2
\end{array}\right],$}\\[0.3cm]

\scalebox{0.5}[0.5]{$%
\left[\begin{array}{rrrrrrr}
-4 & -2 & -1 & -1 & -1 & -1 & -1 \\
-2 & -1 & 2 & 2 & 2 & 2 & 2 \\
-1 & 2 & -4 & 1 & 1 & 1 & 1 \\
-1 & 2 & 1 & -4 & 1 & 1 & 1 \\
-1 & 2 & 1 & 1 & -4 & 1 & 1 \\
-1 & 2 & 1 & 1 & 1 & -4 & 1 \\
-1 & 2 & 1 & 1 & 1 & 1 & -4
\end{array}\right], \left[\begin{array}{rrrrrrr}
-4 & -2 & -1 & -1 & -1 & -1 & -1 \\
-2 & -1 & 2 & 2 & 2 & 2 & 2 \\
-1 & 2 & -4 & 1 & 1 & 1 & 1 \\
-1 & 2 & 1 & -3 & -1 & 0 & 3 \\
-1 & 2 & 1 & -1 & 0 & 3 & -3 \\
-1 & 2 & 1 & 0 & 3 & -3 & -1 \\
-1 & 2 & 1 & 3 & -3 & -1 & 0
\end{array}\right], \left[\begin{array}{rrrrrrr}
-4 & -2 & -1 & -1 & -1 & -1 & -1 \\
-2 & -1 & 2 & 2 & 2 & 2 & 2 \\
-1 & 2 & -3 & -1 & 0 & 1 & 3 \\
-1 & 2 & -1 & 4 & -1 & -1 & -1 \\
-1 & 2 & 0 & -1 & 1 & 3 & -3 \\
-1 & 2 & 1 & -1 & 3 & -3 & 0 \\
-1 & 2 & 3 & -1 & -3 & 0 & 1
\end{array}\right], \left[\begin{array}{rrrrrrr}
-4 & -2 & -1 & -1 & -1 & -1 & -1 \\
-2 & -1 & 2 & 2 & 2 & 2 & 2 \\
-1 & 2 & -1 & -1 & -1 & -1 & 4 \\
-1 & 2 & -1 & -1 & -1 & 4 & -1 \\
-1 & 2 & -1 & -1 & 4 & -1 & -1 \\
-1 & 2 & -1 & 4 & -1 & -1 & -1 \\
-1 & 2 & 4 & -1 & -1 & -1 & -1
\end{array}\right],$}\\[0.3cm]

\scalebox{0.5}[0.5]{$%
\left[\begin{array}{rrrrrrr}
-4 & -2 & -1 & -1 & -1 & -1 & -1 \\
-2 & 3 & -2 & 0 & 0 & 2 & 2 \\
-1 & -2 & 0 & 3 & 3 & 1 & 1 \\
-1 & 0 & 3 & -3 & 2 & 1 & 1 \\
-1 & 0 & 3 & 2 & -3 & 1 & 1 \\
-1 & 2 & 1 & 1 & 1 & -4 & 1 \\
-1 & 2 & 1 & 1 & 1 & 1 & -4
\end{array}\right], \left[\begin{array}{rrrrrrr}
-4 & -2 & -1 & -1 & -1 & -1 & -1 \\
-2 & 3 & -2 & 0 & 0 & 2 & 2 \\
-1 & -2 & 0 & 3 & 3 & 1 & 1 \\
-1 & 0 & 3 & -2 & 1 & -1 & 3 \\
-1 & 0 & 3 & 1 & -2 & 3 & -1 \\
-1 & 2 & 1 & -1 & 3 & 0 & -3 \\
-1 & 2 & 1 & 3 & -1 & -3 & 0
\end{array}\right], \left[\begin{array}{rrrrrrr}
-4 & -2 & -1 & -1 & -1 & -1 & -1 \\
-2 & 3 & -2 & 0 & 0 & 2 & 2 \\
-1 & -2 & 1 & 1 & 3 & 0 & 3 \\
-1 & 0 & 1 & 1 & 2 & 3 & -3 \\
-1 & 0 & 3 & 2 & -3 & 1 & 1 \\
-1 & 2 & 0 & 3 & 1 & -3 & -1 \\
-1 & 2 & 3 & -3 & 1 & -1 & 0
\end{array}\right], \left[\begin{array}{rrrrrrr}
-3 & -3 & -2 & -1 & -1 & -1 & 0 \\
-3 & 0 & 2 & 1 & 1 & 3 & -1 \\
-2 & 2 & -1 & 2 & 2 & -2 & 2 \\
-1 & 1 & 2 & -3 & -1 & 0 & 3 \\
-1 & 1 & 2 & -1 & 0 & -3 & -3 \\
-1 & 3 & -2 & 0 & -3 & 1 & -1 \\
0 & -1 & 2 & 3 & -3 & -1 & 1
\end{array}\right],$}\\[0.3cm]

\scalebox{0.5}[0.5]{$%
\left[\begin{array}{rrrrrrr}
-3 & -3 & -2 & -1 & -1 & -1 & 0 \\
-3 & 1 & 0 & 1 & 2 & 3 & -1 \\
-2 & 0 & 3 & 2 & 0 & -2 & 2 \\
-1 & 1 & 2 & -1 & -3 & 0 & -3 \\
-1 & 2 & 0 & -3 & -1 & 1 & 3 \\
-1 & 3 & -2 & 0 & 1 & -3 & -1 \\
0 & -1 & 2 & -3 & 3 & -1 & -1
\end{array}\right], \left[\begin{array}{rrrrrrr}
-3 & -3 & -2 & -1 & -1 & -1 & 0 \\
-3 & 1 & 2 & -1 & 0 & 3 & -1 \\
-2 & 3 & 0 & -1 & 1 & -3 & 1 \\
-1 & -1 & 0 & 2 & 3 & 1 & 3 \\
-1 & 0 & 2 & 3 & -3 & -1 & 1 \\
-1 & 1 & -2 & 3 & 1 & 0 & -3 \\
0 & -2 & 3 & 0 & 2 & -2 & -2
\end{array}\right], \left[\begin{array}{rrrrrrr}
-3 & -2 & -2 & -2 & -2 & 0 & 0 \\
-2 & -2 & 1 & 2 & 2 & -2 & -2 \\
-2 & 1 & -2 & 2 & 2 & 2 & 2 \\
-2 & 2 & 2 & -2 & 1 & -2 & 2 \\
-2 & 2 & 2 & 1 & -2 & 2 & -2 \\
0 & -2 & 2 & -2 & 2 & 3 & 0 \\
0 & -2 & 2 & 2 & -2 & 0 & 3
\end{array}\right]$}
\end{itemize}

Non-symmetric classes: $6.8,6.9,6.13,7.6,7.18$.
\newpage
\subsection{Automorphism groups, counting, symmetric classes and symmetric automorphisms}
In the following we outline a table (Figure \ref{fig:1})specifying data for each class. The automorphism groups are given by the identifier $p.q$, which is the taxonomy used by the \texttt{GAP} software, as well as in the abstract finite group database \texttt{LMFDB} \cite{lmfdb}. The next column is the class cardinality, followed by the number of symmetric subclasses in each class. The rightmost column specifies the cardinalities of the symmetric automorphism groups of each class. We use there the shorthand $p(n)+q(m)+\cdots$ to say that $p$ subclasses have symmetric automorphism group of cardinality $n$, then $q$ subclasses have symmetric automorphism group of cardinality $m$, etc. The data in the table is sufficient for counting the number of $IW$ and $SIW$, as per \S\ref{sec:6}.

\begin{figure}[h]
    \caption{Automorphism and Symmetric Data for Primitive Classes}\label{fig:1}
{\scriptsize
\begin{center}
\begin{tabular}{|c|c|c|c|c|}
    \hline
    Class Name & Automorphism Group & Class Cardinality & \# Symmetric Subclasses & Symmetric Automorphism Orders\\
    \hline
    1.1 & 2.1 & 2 & 1 & 1(2)\\
    \hline
    2.1 & 4.1 & 16 & 2 & 2(2)\\
    \hline
    4.1 & 16.3 & 9,216 & 5 & 4(8)+(4)\\
    \hline
    4.2 & 8.4 & 18,432 & 1 & (2)\\
    \hline
    5.1 & 8.2 & 1,843,200 & 4 & 4(4)\\
    \hline
    5.2 & 240.189 & 61,440 & 6 & 2(240)+2(24)+2(16)\\
    \hline
    6.1 & 12.1 & 176,947,200 & 2 & 2(4)\\
\hline

6.2 & 12.1 & 176,947,200 & 2 & 1(2)+1(6)\\
\hline

6.3 & 4.1 & 530,841,600 & 1 & 1(2)\\
\hline

6.4 & 32.27 & 66,355,200 & 11 & 1(8)+6(16)+4(32)\\
\hline

6.5 & 24.5 & 88,473,600 & 4 & 2(4)+2(12)\\
\hline

6.6 & 16.3 & 132,710,400 & 5 & 4(8)+1(4)\\
\hline

6.7 & 24.5 & 88,473,600 & 4 & 2(4)+2(12)\\
\hline

6.8 & 8.2 & 265,420,800 & 0 & -- \\
\hline

6.9 & 16.3 & 132,710,400 & 0 & -- \\
\hline

6.10 & 144.115 & 14,745,600 & 5 & 1(8)+2(24)+1(72)+1(12)\\
\hline

6.11 & 12.1 & 176,947,200 & 1 & 1(6)\\
\hline

6.12 & 48.30 & 44,236,800 & 3 & 1(8)+1(24)+1(4)\\
\hline

6.13 & 12.1 & 176,947,200 & 0 & --\\
\hline

7.1 & 8.2 & 52,022,476,800 & 4 & 4(4)\\
\hline

7.2 & 4.2 & 104,044,953,600 & 4 & 4(4)\\
\hline

7.3 & 24.14 & 17,340,825,600 & 8 & 4(8)+4(24)\\
\hline

7.4 & 48.51 & 8,670,412,800 & 8 & 4(16)+4(48)\\
\hline

7.5 & 6.2 & 69,363,302,400 & 2 & 2(2)\\
\hline

7.6 & 2.1 & 208,089,907,200 & 0 & -- \\
\hline

7.7 & 4.2 & 104,044,953,600 & 4 & 4(4)\\
\hline

7.8 & 2.1 & 208,089,907,200 & 2 & 2(2)\\
\hline

7.9 & 1440.5842 & 289,013,760 & 8 & 2(32)+4(96)+2(1440)\\
\hline

7.10 & 16.10 & 26,011,238,400 & 8 & 8(8)\\
\hline

7.11 & 8.2 & 52,022,476,800 & 4 & 4(4)\\
\hline

7.12 & 240.189 & 1,734,082,560 & 6 & 2(16)+2(24)+2(240)\\
\hline

7.13 & 24.14 & 17,340,825,600 & 8 & 4(8)+4(24)\\
\hline

7.14 & 4.2 & 104,044,953,600 & 4 & 4(4)\\
\hline

7.15 & 2.1 & 208,089,907,200 & 2 & 2(2)\\
\hline

7.16 & 6.2 & 69,363,302,400 & 2 & 2(2)\\
\hline

7.17 & 8.2 & 52,022,476,800 & 4 & 4(4)\\
\hline

7.18 & 8.2 & 52,022,476,800 & 0 & -- \\
\hline

7.19 & 48.48 & 8,670,412,800 & 6 & 2(8)+2(16)+2(48)\\
\hline

\end{tabular}
\end{center}
}%
\end{figure}

\subsection{Non-primitive classification}

In the next table , Figure \ref{fig:2}, we organize all of the possible primitive decompositions of an $IW(7,25)$. These are the partitions of $7$ to smaller blocks of available sizes. An entry $A\oplus 3B$ means that we take a block of type $A$ and 3 blocks of type $B$. We list down the multiplicities up to H and TH-equivalences, which are derived from the number of possible primitive classes of each type. Notice that in each matrix there can be at most one block which is not of symmetric type, making the TH analysis straightforward.

\begin{center}
    \begin{figure}[h]
        \caption{The full H and TH classification of $IW(7,25)$} \label{fig:2}
\begin{tabular}{|c|c|c|}
    \hline
    Decomposition Type & TH-Multiplicity & H-multiplicity\\
    \hline 
    $7A$ & 1 & 1\\
    \hline
    $5A\oplus B$ & 1& 1\\
    \hline
    $3A\oplus 2B$ & 1 & 1\\
    \hline
    $A\oplus 3B$ & 1 & 1\\
    \hline
    $3A\oplus C$ & 2 & 2\\
    \hline
    $A\oplus B\oplus C$ & 2 & 2\\
    \hline
    $2A\oplus D$ &2 & 2\\ 
    \hline
    $B\oplus D$ & 2 & 2\\
    \hline
    $A\oplus E$ & 13 & 16\\
    \hline
    $F$ & 19 & 21\\
    \hline
    \hline
    Total & 44 & 49\\
    \hline
\end{tabular}
\end{figure}
\end{center}

\subsection{Total counts}
Total counts: $|IW(7,25)|=1,915,159,357,440$, $|SIW(7,25)|=14,813,808$, and of course $|AIW(7,25)|=0$. The number of primitive $IW(7,25)$ is $1,623,390,289,920$. A random $IW(7,25)$ has probability of $84.8\%$ to be primitive.

\bigskip


\bibliography{weighing}

\end{document}